\documentclass[12pt]{amsart}
\pdfoutput=1

\usepackage{tikz, amssymb, microtype}

\usepackage[margin=1in,letterpaper,portrait]{geometry}

\usepackage{hyperref}

\theoremstyle{definition}
\newtheorem{thm}{Theorem}[section]
\newtheorem{prop}[thm]{Proposition}

\newtheorem{conj}[thm]{Conjecture}
\newtheorem{lemma}[thm]{Lemma}
\newtheorem{defn}[thm]{Definition}
\newtheorem{cor}[thm]{Corollary}
\newtheorem{obs}[thm]{Observation}
\newtheorem{question}[thm]{Question}
\newtheorem{rem}[thm]{Remark}

\DeclareMathOperator{\Des}{Des}

\DeclareMathOperator{\Ret}{Ret}
\DeclareMathOperator{\ret}{ret}

\DeclareMathOperator{\des}{des}

\DeclareMathOperator{\op}{\Omega}

\DeclareMathOperator{\A}{\mathcal{A}}
\DeclareMathOperator{\LL}{\mathcal{L}}
\DeclareMathOperator{\OO}{\mathcal{O}}
\DeclareMathOperator{\CC}{\mathcal{C}}

\DeclareMathOperator{\wt}{wt}

\newcommand{\RR}{\mathbb{R}}

\newcommand{\NN}{\mathbb{N}}

\let\originalleft\left
\let\originalright\right
\renewcommand{\left}{\mathopen{}\mathclose\bgroup\originalleft}
\renewcommand{\right}{\aftergroup\egroup\originalright}

\title{Zig-zag Eulerian polynomials}

\author[Petersen]{T. Kyle Petersen}
\address{Department of Mathematical Sciences, DePaul University, Chicago, IL}
\email{tpeter21@depaul.edu}

\author[Zhuang]{Yan Zhuang}
\address{Department of Mathematics and Computer Science, Davidson College, Davidson, NC}
\email{yazhuang@davidson.edu}

\date{\today}



\begin{document}

\begin{abstract}
For any finite partially ordered set $P$, the $P$-Eulerian polynomial is the generating function for the descent number over the set of linear extensions of $P$, and is closely related to the order polynomial of $P$ arising in the theory of $P$-partitions. Here we study the $P$-Eulerian polynomial where $P$ is a naturally labeled zig-zag poset; we call these \emph{zig-zag Eulerian polynomials}. A result of Br\"and\'en implies that these polynomials are gamma-nonnegative, and hence their coefficients are symmetric and unimodal. The zig-zag Eulerian polynomials and the associated order polynomials have appeared fleetingly in the literature in a wide variety of contexts---e.g., in the study of polytopes, magic labelings of graphs, and Kekul\'e structures---but they do not appear to have been studied systematically. 

In this paper, we use a ``relaxed'' version of $P$-partitions to both survey and unify results. Our technique shows that the zig-zag Eulerian polynomials also capture the distribution of ``big returns'' over the set of (up-down) alternating permutations, as first observed by Coons and Sullivant. We develop recurrences for refined versions of the relevant generating functions, which evoke similarities to recurrences for the classical Eulerian polynomials. We conclude with a literature survey and open questions.
\end{abstract}

\dedicatory{Dedicated to our Ph.D. advisor Ira M.\ Gessel, on the occasion of his 73rd birthday.}

\maketitle

\section{Introduction}

The theory of $P$-partitions has a long history with many applications to algebraic and enumerative combinatorics; see \cite[Chapter 3]{EC1} for a comprehensive introduction and \cite{Gessel2016} for a historical survey of this topic. There are two important polynomials that arise in this theory: the \emph{order polynomial} $\op_P(m)$, which counts bounded $P$-partitions, and the \emph{$P$-Eulerian polynomial} $A_P(t)$, which counts the linear extensions of $P$ according to descent number. These two polynomials are related by the formula
\begin{equation}
 \frac{A_P(t)}{(1-t)^{n+1}}= \sum_{m\geq 1} \op_P(m) t^m. \label{e-Eulord}
\end{equation}
Let $S_n$ denote the symmetric group on $[n]=\{1,2,\ldots,n\}$. Recall that a \emph{descent} of a permutation $w\in S_n$ is a position $i\in[n-1]$ such that $w(i)>w(i+1)$. We denote the descent set by 
\[
\Des(w) = \{\, i\in[n-1] : w(i)>w(i+1)\,\}
\]
and the number of descents by $\des(w)=\left|\Des(w)\right|$. The classical Eulerian polynomial
\[ 
A_n(t)=\sum_{w \in S_n} t^{1+\des(w)}
\]
is the $P$-Eulerian polynomial when $P$ is an antichain of $n$ elements, for which the order polynomial is $m^n$. So, in this case \eqref{e-Eulord} is the well-known identity
\begin{equation}\label{eq:classic}
 \frac{ A_n(t)}{(1-t)^{n+1}} = \frac{\sum_{w \in S_n} t^{1+\des(w)}}{(1-t)^{n+1}} =\sum_{m \geq 0} m^n t^m,
\end{equation}
which is sometimes used to define the Eulerian polynomials.

In this paper, we focus on a sequence of posets commonly known as the \emph{zig-zag posets}; see Figure \ref{fig:Zn}. Taking $P$ to be any natural labeling of the zig-zag poset on $n$ elements, we call $Z_n(t)=A_{P}(t)$ the $n$th \emph{zig-zag Eulerian polynomial} and $\op_n(m)=\op_{P}(m)$ the $n$th \emph{zig-zag order polynomial}. These polynomials have appeared in the literature since at least the 1970s in a wide variety of contexts---including the study of polytopes, magic labelings of graphs, and even Kekul\'e structures in chemistry---although more attention has been given to $\op_n(m)$ than to $Z_n(t)$. The study of these polynomials by different authors over the past decades seem to have occurred mostly independently, with many of the authors evidently unaware of each others' work, and so one of the goals of our present paper is to give a unified treatment of these polynomials using the theory of $P$-partitions. 

As first shown by Coons and Sullivant \cite{CoonsSullivant2023}, the zig-zag Eulerian polynomials also encode the distribution of the ``big return'' statistic over (up-down) alternating permutations. In other words, $Z_n(t) = tU_n(t)$, where $U_n(t)$ is the generating function for big returns of alternating permutations. We will give a proof of this equidistribution using a modified version of $P$-partitions that we call ``relaxed'' $P$-partitions. Moreover, Coons and Sullivant demonstrated that this distribution is symmetric and unimodal by proving that the order polytope of the zig-zag poset is Gorenstein. We shall instead use a result of Br\"and\'en to show that these polynomials are gamma-nonnegative, which implies both symmetry and unimodality of their coefficients.

Having laid some of this groundwork, we undertake the task of investigating generating functions and recurrences for the polynomials $Z_n(t)$ and $\op_n(m)$. To explain one of our main results, we first recall an analogous result for classical Eulerian polynomials. Letting $H_n(t) = A_n(t)/(1-t)^{n+1}$, it is easy to see from the right-hand side of \eqref{eq:classic} that
\begin{equation}\label{eq:Hrec}
 tH_n'(t) = \sum_{m\geq 1} m^{n+1}t^m = H_{n+1}(t),
\end{equation}
which in turn implies
\begin{equation}\label{eq:Arec}
A_{n+1}(t) = t(1-t)A_n'(t)+(n+1)tA_n(t).
\end{equation}
Equation \eqref{eq:Arec} is one of the most fundamental recursive formulas for the Eulerian polynomials. This recursion allows easy proof of, for example, the fact that the Eulerian polynomials have only real roots, which in turn allows one to prove that the Eulerian distribution is asymptotically normal \cite{Bender}.

Define $G_n(t) = Z_n(t)/(1-t)^{n+1}$ to be the zig-zag analogue of $H_n(t)$. In Section \ref{sec:refinement}, we introduce a refinement of this generating function that is best expressed as
\[
 G_n(p,q,x) = \sum_{m\geq 1} \op_n(p,q;m)x^m = \frac{pqxU_n(px,qx)}{(1-qx)^{n+1}},
\]
where $U_n(s,t)$ is a refinement of the big return generating function, i.e., $tU_n(t,t)=tU_n(t)= Z_n(t)$. At $p=q=1$, we have $G_n(1,1,x) = G_n(x) = Z_n(x)/(1-x)^{n+1}$.  

We define $G_0(p,q,x) = pqx/(1-qx)$ and $U_0(s,t) = 1$. The following result provides one recursive procedure for computing $G_n(p,q,x)$ and $U_n(s,t)$ for larger $n$.

\begin{thm}[Refined recurrences]\label{thm:refined}
For all $n\geq 0$,
\begin{itemize}
\item[(a)] $\displaystyle G_{n+1}(p,q,x) = \frac{p}{p-q}\left[ G_n(q,p,x) - G_n(q,q,x) \right],$ and thus
\vspace{5bp}
\item[(b)] $\displaystyle U_{n+1}(s,t) = \frac{1-t}{s-t}\left[ \left(\frac{1-t}{1-s}\right)^{n+1} sU_n(t,s) - tU_n(t,t) \right].$
\end{itemize}
\end{thm}

By taking the appropriate limits of parts (a) and (b) of Theorem \ref{thm:refined}, we obtain the following corollary, which provides expressions analogous to Equations \eqref{eq:Hrec} and \eqref{eq:Arec}.

\begin{cor}[Derivative expressions]\label{cor:deriv}
For all $n\geq 0$,
\begin{itemize}
\item[(a)] $\displaystyle G_{n+1}(x) = \frac{d}{dq}\left[ G_n(1,q,x) \right]_{q=1},$
\vspace{5bp}
\item[(b1)] $\displaystyle U_{n+1}(t) = t(1-t)\frac{d}{dt}\left[ U_n(s,t) \right]_{s=t} + (nt+1)U_n(t)$, and
\vspace{5bp}
\item[(b2)] $\displaystyle Z_{n+1}(t) = t(1-t)\frac{d}{dt}\left[ tU_n(s,t) \right]_{s=t} + (n+1)tZ_n(t)$.
\end{itemize}
\end{cor}

This paper is organized as follows. In Section \ref{s-bg}, we provide a brief expository overview of $P$-partition theory and zig-zag posets, and we apply Br\"and\'en's result to establish the gamma-nonnegativity of $Z_n(t)$. In Section \ref{sec:alt}, we recall some properties of alternating permutations and define big returns, our key statistic of interest. Section \ref{s-relaxed} is devoted to our modification of the theory of $P$-partitions, and in Section \ref{sec:proof1} we use this theory to show that the polynomials $Z_n(t)$ count alternating permutations by big returns.
In Section \ref{s-F}, we use our approach to derive both new and previously-known recurrences for $\op_n(m)$; see Theorem \ref{thm:Frec}. We introduce our refined order polynomial in Section \ref{sec:refinement} and use it to prove Theorem \ref{thm:refined} and Corollary \ref{cor:deriv}. Section \ref{s-interp} summarizes other known combinatorial interpretations of the zig-zag Eulerian polynomials and their order polynomials, and we end in Section \ref{s-conclusion} with open questions for further study.

\section{$P$-partitions and zig-zag posets} \label{s-bg}

Let us begin by providing the necessary background from the theory of $P$-partitions. All results asserted in this section can be found in \cite[Section 3.15]{EC1}.

\subsection{$P$-partitions}

Throughout, we fix a positive integer $n$ and let $P$ denote a partial ordering of the set $[n]=\{1,2,\ldots,n\}$. We write ``$<_P$'' for the order relation on $P$, i.e., if $i$ is below $j$ in $P$, we say $i$ and $j$ are \emph{comparable} and write $i <_P j$. If neither $i<_P j$ nor $j <_P i$, we say $i$ and $j$ are \emph{incomparable}. A comparable pair $i<_P j$ is \emph{naturally labeled} if $i < j$ in the integers as well; otherwise, the pair is \emph{unnaturally labeled}. 

A \emph{chain} is a poset in which any two elements are comparable. We readily identify chains with permutations, via
\[
 \pi(1) <_{\pi} \pi(2) <_{\pi} \cdots <_{\pi} \pi(n)
\]
whenever $\pi \in S_n$. To say that $i <_{\pi} j$ is equivalent to saying that $\pi^{-1}(i) < \pi^{-1}(j)$ as integers.

We say that a poset $Q$ \emph{refines} $P$ if every relation in $P$ is a relation in $Q$. That is, $Q$ refines $P$ if $i <_P j$ implies $i <_Q j$. In this setting, chains are maximally refined posets. We define the set $\LL(P)$ of \emph{linear extensions} of $P$ to be the set of chains (permutations) that refine $P$:
\[
 \LL(P) = \{\, \pi \in S_n : i <_P j \Rightarrow  i <_{\pi} j \,\}.
\]

We now review the definition of a $P$-partition.

\begin{defn}[$P$-partition]
A \emph{$P$-partition} is an order-preserving function $f: P \to \NN = \{1,2,3,\ldots\}$ such that for $i <_P j$:
\begin{itemize}
\item $f(i) \leq f(j)$ if $i < j$,
\item $f(i) < f(j)$ if $i > j$.
\end{itemize}
\end{defn}

In other words, the values of a $P$-partition on a naturally labeled pair are allowed to agree, while the values on an unnaturally labeled pair must be distinct. We denote the set of $P$-partitions by $\A(P)$.

For example, if $n=3$ and $P$ is the poset with $1 <_P 2$ and $3 <_P 2$, then $P$ has linear extensions $\LL(P)=\{132, 312\}$. We can draw the poset and its extensions with \emph{Hasse diagrams} as indicated here:
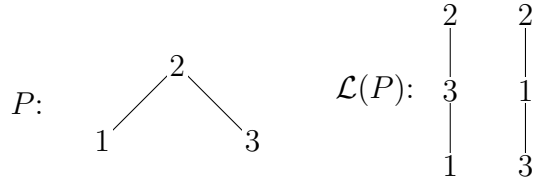
\[
 \begin{tikzpicture}[baseline=.5cm]
  \draw (0,1.5) node {$P$:};
  \draw (1,1) node[fill=white,inner sep=1] {$1$}-- (2,2) node[fill=white,inner sep=1] {$2$} -- (3,1) node[fill=white,inner sep=1] {$3$};
 \end{tikzpicture}
 \qquad
 \begin{tikzpicture}
  \draw (0,1) node {$\LL(P)$:};
  \draw (1,0) node[fill=white,inner sep=1] {$1$}-- (1,1) node[fill=white,inner sep=1] {$3$} -- (1,2) node[fill=white,inner sep=1] {$2$};
  \draw (2,0) node[fill=white,inner sep=1] {$3$}-- (2,1) node[fill=white,inner sep=1] {$1$} -- (2,2) node[fill=white,inner sep=1] {$2$};
 \end{tikzpicture}
\]
In this case, every $P$-partition $f$ must satisfy
\[
 f(1)\leq f(2) > f(3),
\]
so we have
\[
 \A(P) = \{\, (a_1, a_2, a_3)\in  \NN^3 : a_1 \leq a_2 > a_3 \,\}
\]
where we are identifying each $P$-partition $f$ with the values $(f(1),f(2),f(3))$. We can write $\A(P)$ as the disjoint union
\[
 \A(P) = \{ a_1 \leq a_3 < a_2 \} \cup \{ a_3 < a_1 \leq a_2 \},
\]
where each of these smaller sets can viewed as the set of $P$-partitions for a chain:
\[
 \A(P) = \A(132) \cup \A(312).
\]

By induction on the number of incomparable pairs in a poset $P$, we can see that the set of all $P$-partitions is the disjoint union of the $P$-partitions for its linear extensions. This is sometimes known as the \emph{fundamental lemma of $P$-partitions} \cite[Lemma 3.15.3]{EC1}.

\begin{lemma}[Fundamental lemma of $P$-partitions]\label{thm:ftpp}
The set of $P$-partitions is the disjoint union of the $\pi$-partitions of its linear extensions $\pi$:
\[
 \A(P) = \bigcup_{\pi \in \LL(P)} \A(\pi).
\]
\end{lemma}

The set of $P$-partitions bounded by $m$ is denoted $\A(P;m)$, i.e.,
\[ 
 \A(P;m) = \{\, f \in \A(P) : 1\leq f(i) \leq m \text{ for all }i \,\}.
\]
The \emph{order polynomial} of $P$ is the cardinality of $\A(P;m)$, denoted $\op_P(m)=\left|\A(P;m)\right|$. As a consequence of Lemma \ref{thm:ftpp}, we have
\[
 \op_P(m) = \sum_{\pi \in \LL(P)} \op_{\pi}(m).
\]
 
In the case of a permutation $\pi$, we have $f \in \A(\pi)$ if and only if
\[
 f(\pi(1))\leq f(\pi(2))\leq \cdots \leq f(\pi(n)),
\]
with $f(\pi(i)) < f(\pi(i+1))$ when $i \in \Des(\pi)$. Thus 
\begin{align}
 \op_{\pi}(m) &= |\{\, 1\leq a_1 \leq a_2 \leq \cdots \leq a_n \leq m : i \in \Des(\pi) \Rightarrow a_i < a_{i+1} \,\}| \nonumber \\
 &= \binom{m+n-1-\des(\pi)}{n} \label{eq:piop}
\end{align}
(see \cite[Chapter 3]{EC1}), and therefore
\[
 \sum_{m\geq 0} \op_{\pi}(m) t^m = \sum_{m\geq 0} \binom{m+n-1-\des(\pi)}{n} t^m = \frac{t^{1+\des(\pi)}}{(1-t)^{n+1}}.
\]

Combining this expression with Lemma \ref{thm:ftpp}, we see that the order polynomial generating function is a rational function of the form
\begin{equation}\label{eq:ogf}
 \frac{A_P(t)}{(1-t)^{n+1}} = \sum_{m\geq 0} \op_P(m) t^m,
\end{equation}
where $A_P(t)$ is a polynomial of degree at most $n$, called the \emph{$P$-Eulerian polynomial}, given by
\[
 A_P(t) = \sum_{w \in \LL(P)} t^{1+\des(w)}.
\]

\subsection{Zig-zag posets}

We now consider, for each $n\geq 1$, a poset $Z_n$ that we call the \emph{zig-zag poset}. This poset $Z_n$ is defined by the relations $1 <_P 2 >_P 3 <_P \cdots$, i.e., $2i-1<_P 2i >_P 2i+1$ for $1\leq i \leq (n-1)/2$, with $n-1 <_P n$ if $n$ is even. 

Posets on $[n]$ admit a left $S_n$-action by permuting labels. Let $s_i = (i,i+1)$ denote the adjacent transposition that swaps $i$ and $i+1$, and let $\varepsilon = s_2 s_4 \cdots s_{2k}$ be the product of all even simple transpositions, with $k=\lfloor n/2 \rfloor$. Then it is straightforward to check that the poset $\varepsilon Z_n$ is naturally labeled. Figure \ref{fig:Zn} shows the Hasse diagrams of $Z_6$, $\varepsilon Z_6$, $Z_7$, and $\varepsilon Z_7$.

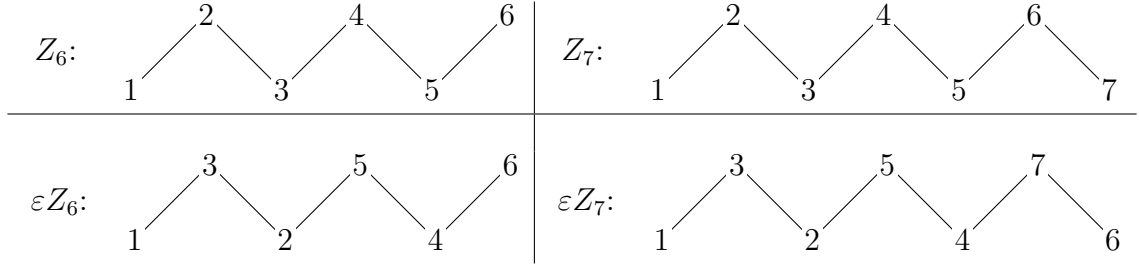
\begin{figure}
\[
\begin{array}{c|c}
\begin{tikzpicture}
  \draw (0,.5) node {$Z_6$:};
  \draw (1,0) node[fill=white,inner sep=1] {$1$}-- (2,1) node[fill=white,inner sep=1] {$2$} -- (3,0) node[fill=white,inner sep=1] {$3$} -- (4,1) node[fill=white,inner sep=1] {$4$}-- (5,0) node[fill=white,inner sep=1] {$5$}-- (6,1) node[fill=white,inner sep=1] {$6$};
 \end{tikzpicture}
 &
 \begin{tikzpicture}
  \draw (0,.5) node {$Z_7$:};
  \draw (1,0) node[fill=white,inner sep=1] {$1$}-- (2,1) node[fill=white,inner sep=1] {$2$} -- (3,0) node[fill=white,inner sep=1] {$3$} -- (4,1) node[fill=white,inner sep=1] {$4$}-- (5,0) node[fill=white,inner sep=1] {$5$}-- (6,1) node[fill=white,inner sep=1] {$6$}-- (7,0) node[fill=white,inner sep=1] {$7$};
 \end{tikzpicture}
\\
\hline
\\
\begin{tikzpicture}
  \draw (0,.5) node {$\varepsilon Z_6$:};
  \draw (1,0) node[fill=white,inner sep=1] {$1$}-- (2,1) node[fill=white,inner sep=1] {$3$} -- (3,0) node[fill=white,inner sep=1] {$2$} -- (4,1) node[fill=white,inner sep=1] {$5$}-- (5,0) node[fill=white,inner sep=1] {$4$}-- (6,1) node[fill=white,inner sep=1] {$6$};
 \end{tikzpicture}
 &
 \begin{tikzpicture}
  \draw (0,.5) node {$\varepsilon Z_7$:};
  \draw (1,0) node[fill=white,inner sep=1] {$1$}-- (2,1) node[fill=white,inner sep=1] {$3$} -- (3,0) node[fill=white,inner sep=1] {$2$} -- (4,1) node[fill=white,inner sep=1] {$5$}-- (5,0) node[fill=white,inner sep=1] {$4$}-- (6,1) node[fill=white,inner sep=1] {$7$}-- (7,0) node[fill=white,inner sep=1] {$6$};
 \end{tikzpicture}
\end{array}
\]
\caption{The Hasse diagrams for the zig-zag posets $Z_6$ and $Z_7$, and the naturally labeled posets $\varepsilon Z_6$ and $\varepsilon Z_7$.}\label{fig:Zn}
\end{figure}

Our primary objects of study in this paper are the order polynomials and Eulerian polynomials for $\varepsilon Z_n$. To simplify notation, we write 
\[
 Z_n(t) = A_{\varepsilon Z_n}(t) \quad \mbox{ and } \quad \op_n(m) = \op_{\varepsilon Z_n}(m),
\]
which we refer to as the $n$th \emph{zig-zag Eulerian polynomial} and \emph{zig-zag order polynomial}, respectively. We denote by $z(n,k)$ the coefficient of $t^{k+1}$ in $Z_n(t)$, so that
\[
 z(n,k) = |\{\, w \in \LL(\varepsilon Z_n) : \des(w) = k \,\}|;
\]
we call these \emph{zig-zag Eulerian numbers}.

For example, we see the linear extensions of $\varepsilon Z_5$ in Table \ref{tab:zigex}, grouped according to the number of descents. From this we conclude $z(5,0)=z(5,3)=1$ and $z(5,1)=z(5,2)=7$, so that $Z_5(t) = t+7t^2+7t^3+t^4$. Therefore, from Equation \eqref{eq:ogf} we find
\begin{equation}\label{eq:G5}
 \frac{ t+7t^2+7t^3+t^4}{(1-t)^6} = \sum_{m\geq 0} \op_n(m) t^m = t+13t^2 + 70t^3 + 246t^4 + 671t^5+\cdots .
\end{equation}
Since $(1-t)^{-6} = \sum_{m\geq 0} \binom{m+5}{5} t^m$, we have
\begin{align*}
 \op_5(m) &= \binom{m+4}{5} + 7\binom{m+3}{5} + 7\binom{m+2}{5} + \binom{m+1}{5} \\
 &= \frac{4m+20m^2+40m^3+40m^4+16m^5}{5!}.
\end{align*}

\begin{table}
\[
\begin{array}{c | c | c | c}
 \des(w)=0 & \des(w)=1 & \des(w)=2 & \des(w)=3 \\
 \hline
 \hline
 12345 & 124|35 & 14|25|3 & 4|2|15|3 \\
       & 1245|3 & 2|14|35 & \\
       & 14|235 & 2|145|3 & \\
       & 2|1345 & 24|15|3 & \\
       & 24|135 & 4|125|3 & \\
       & 245|13 & 4|2|135 & \\
       & 4|1235 & 4|25|13 & 
\end{array}
\]
\caption{The linear extensions in $\LL(\varepsilon Z_5)$, grouped according to the number of descents, which are marked with vertical bars.}\label{tab:zigex}
\end{table}

Table \ref{tab:znk} shows the values of $z(n,k)$ for $n\leq 10$, and Table \ref{tab:opnk} shows values of $\op_n(m)$ for $n\leq 10$ and $m\leq 8$. (We have defined $\op_0(m)=1$ for convenience.) The array of numbers $\op_n(m)$ can be found in entry A050446 of the OEIS, while the numbers $z(n,k)$ are the antidiagonals of A205497 \cite{oeis}.

\begin{table}
\[
 \begin{array}{r|ccccccccc}
  n\backslash k & 0 & 1 & 2 & 3 & 4 & 5 & 6 & 7 & 8 \\
  \hline 
  1 & 1 \\
  2 & 1 \\
  3 & 1 & 1 \\
  4 & 1 & 3 & 1 \\
  5 & 1 & 7 & 7 & 1\\
  6 & 1 & 14 & 31 & 14 & 1 \\
  7 & 1 & 26 & 109 & 109 & 26 & 1 \\
  8 & 1 & 46 & 334 & 623 & 334 & 46 & 1 \\
  9 & 1 & 79 & 937 & 2951 & 2951 & 937 & 79 & 1 \\
  10 & 1 & 133 & 2475 & 12331 & 20641 & 12331 & 2475 & 133 & 1
 \end{array}
\]
\caption{Table of values of $z(n,k)$ for $1\leq n \leq 10$ and $0\leq k \leq n-2$.}\label{tab:znk}
\end{table}

\begin{table}
\[
 \begin{array}{r|ccccccccc}
  n\backslash m & 1 & 2 & 3 & 4 & 5 & 6 & 7 & 8 \\
  \hline 
  0 & 1 & 1 & 1 & 1 & 1 & 1 & 1 & 1 \\
  1 & 1 & 2 & 3 & 4 & 5 & 6 & 7 & 8 \\
  2 & 1 & 3 & 6 & 10 & 15 & 21 & 28 & 36 \\
  3 & 1 & 5 & 14 & 30 & 55 & 91 & 140 & 204 \\
  4 & 1 & 8 & 31 & 85 & 190 & 371 & 658 & 1086 \\
  5 & 1 & 13 & 70 & 246 & 671 & 1547 & 3164 & 5916\\
  6 & 1 & 21 & 157 & 707 & 2353 & 6405 & 15106 & 31998 \\
  7 & 1 & 34 & 353 & 2037 & 8272 & 26585 & 72302 & 173502 \\
  8 & 1 & 55 & 793 & 5864 & 29056 & 110254 & 345775 & 940005 \\
  9 & 1 & 89 & 1782 & 16886 & 102091 & 457379 & 1654092 & 5094220 \\
  10 & 1 & 144 & 4004 & 48620 & 358671 & 1897214 & 7911970 & 27604798
 \end{array}
\]
\caption{Table of values of $\op_n(m)$ for $0\leq n \leq 10$ and $1\leq m \leq 8$.}\label{tab:opnk}
\end{table}

\begin{rem}[Different natural labelings]
A different natural labeling of the zig-zag poset, $\omega Z_n$, appears in  \cite[Example 3.2]{Stanley1980} and \cite[Exercise 3.66]{EC1}. With $k=\lceil n/2 \rceil$, this poset is defined by the relations $i <_P (k+i) >_P i+1$ for $1 \leq i \leq k-1$, and $k <_P 2k$ if $n=2k$. The linear extensions of this poset are not the same as the linear extensions of the poset $\varepsilon Z_n$. However, the $P$-partitions of any naturally labeled poset  are just the order-preserving maps $P\to \NN$, so $\A(\omega Z_n) = \A(\varepsilon Z_n)$ and the two labeled posets have the same order polynomial and the same $P$-Eulerian polynomial.
\end{rem}

\subsection{Up-down versus down-up}

We have defined the ``up-down'' zig-zag poset $Z_n$, but we could also study the ``down-up'' zig-zag poset  $1>_P 2 <_P 3 >_P \cdots$, which we denote $\bar Z_n$. Letting $o = s_1 s_3 \cdots s_{2k-1}$, we see that $o \bar Z_n$ is naturally labeled; see Figure \ref{fig:barZn}. We claim that the up-down and down-up versions of the order polynomial and Eulerian polynomial are the same.

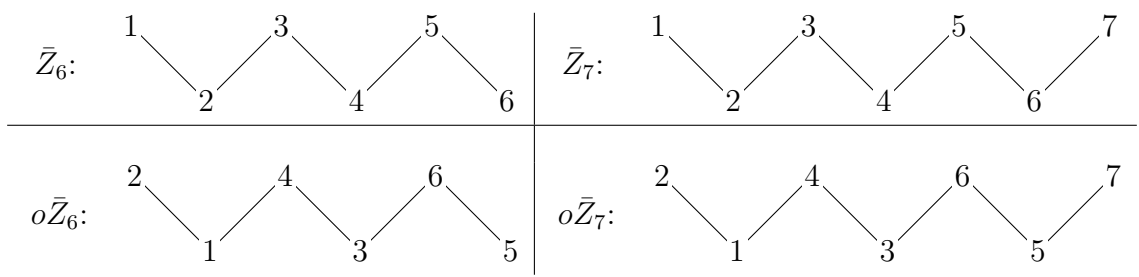
\begin{figure}
\[
\begin{array}{c|c}
\begin{tikzpicture}
  \draw (0,.5) node {$\bar Z_6$:};
  \draw (1,1) node[fill=white,inner sep=1] {$1$}-- (2,0) node[fill=white,inner sep=1] {$2$} -- (3,1) node[fill=white,inner sep=1] {$3$} -- (4,0) node[fill=white,inner sep=1] {$4$}-- (5,1) node[fill=white,inner sep=1] {$5$}-- (6,0) node[fill=white,inner sep=1] {$6$};
 \end{tikzpicture}
 &
 \begin{tikzpicture}
  \draw (0,.5) node {$\bar Z_7$:};
  \draw (1,1) node[fill=white,inner sep=1] {$1$}-- (2,0) node[fill=white,inner sep=1] {$2$} -- (3,1) node[fill=white,inner sep=1] {$3$} -- (4,0) node[fill=white,inner sep=1] {$4$}-- (5,1) node[fill=white,inner sep=1] {$5$}-- (6,0) node[fill=white,inner sep=1] {$6$}-- (7,1) node[fill=white,inner sep=1] {$7$};
 \end{tikzpicture}
\\
\hline
\\
\begin{tikzpicture}
  \draw (0,.5) node {$o \bar Z_6$:};
  \draw (1,1) node[fill=white,inner sep=1] {$2$}-- (2,0) node[fill=white,inner sep=1] {$1$} -- (3,1) node[fill=white,inner sep=1] {$4$} -- (4,0) node[fill=white,inner sep=1] {$3$}-- (5,1) node[fill=white,inner sep=1] {$6$}-- (6,0) node[fill=white,inner sep=1] {$5$};
 \end{tikzpicture}
 &
 \begin{tikzpicture}
  \draw (0,.5) node {$o \bar Z_7$:};
  \draw (1,1) node[fill=white,inner sep=1] {$2$}-- (2,0) node[fill=white,inner sep=1] {$1$} -- (3,1) node[fill=white,inner sep=1] {$4$} -- (4,0) node[fill=white,inner sep=1] {$3$}-- (5,1) node[fill=white,inner sep=1] {$6$}-- (6,0) node[fill=white,inner sep=1] {$5$}-- (7,1) node[fill=white,inner sep=1] {$7$};
 \end{tikzpicture}
\end{array}
\]
\caption{The Hasse diagrams for the down-up zig-zag posets $\bar Z_6$ and $\bar Z_7$, and the naturally labeled posets $o \bar Z_6$ and $o \bar Z_7$.}\label{fig:barZn}
\end{figure}

\begin{thm}\label{thm:ZbarZ}
For all $n \geq 1$,
\begin{equation}
 \op_n(m) = \op_{o \bar Z_n}(m) \label{e-opuddu} 
\end{equation}
and hence
\[
 Z_n(t) = A_{o \bar Z_n}(t).
\]
\end{thm}

\begin{proof}
We will construct an explicit bijection between $\A(\varepsilon Z_n; m)$ and $\A(o \bar Z_n; m)$, from which the results follow.

Let $f \in \A(\varepsilon Z_n; m)$, i.e.,
\[
 f(1) \leq f(3) \geq f(2) \leq f(5) \geq \cdots.
\]
Then we see 
\[
f(\varepsilon(1)) \leq f(\varepsilon(2)) \geq f(\varepsilon(3)) \leq f(\varepsilon(4)) \geq \cdots,
\]
and $f\varepsilon$ is an order-preserving map $Z_n \to [m]$.

Let $g$ be defined by $g(i) = m+1-(f\varepsilon)(i)$, so that
\[
 g(1) \geq g(2) \leq g(3) \geq g(4) \leq \cdots,
\]
and we see that $g$ is an order-preserving map $\bar Z_n \to [m]$. But then
\[
 g(o(2)) \geq g(o(1)) \leq g(o(4)) \geq g(o(3)) \leq \cdots,
\]
i.e., $go$ is an element of $\A(o \bar Z_n; m)$.

The correspondence $f \leftrightarrow go$ is clearly bijective, and the result follows.
\end{proof}

\subsection{Gamma-nonnegativity}\label{sec:gamma}

It has been remarked several times in the literature that the coefficients of $Z_n(t)$ are symmetric and unimodal, i.e., $z(n,k) = z(n,n-2-k)$ and 
\[
 z(n,0)\leq z(n,1) \leq \cdots \leq z(n,\lfloor (n-2)/2 \rfloor) \geq \cdots \geq z(n,n-2).
\]
In particular, unimodality was conjectured by Kirillov \cite[Conjecture 3.11]{Kirillov2000} and proved by Chen and Zhang \cite{ChenZhang2016}, and also by Coons and Sullivant \cite[Theorem 32]{CoonsSullivant2023}. (All these authors were working in the context of simplicial polytopes, in which symmetry follows from the Dehn--Sommerville equations.) Here we will make a stronger claim, that $Z_n(t)$ is \emph{gamma-nonnegative}. Let us recall this property.

Suppose $h(t) = \sum_i h_i t^i$ is a polynomial with integer coefficients such that $t^d h(1/t) = h(t)$ for some positive integer $d$.  Then the coefficients of $h$ are symmetric---i.e., $h_i = h_{d-i}$ for all $i$---and hence we will say that $h$ is \emph{palindromic}. It is easy to verify that if such a $d$ exists, it is the sum of largest and smallest powers of $t$ with nonzero coefficients in $h$. We define this integer $d$ to be the \emph{palindromic degree} of $h$. 

Palindromic polynomials of palindromic degree $d$ span a vector space of dimension roughly $d/2$, and with this in mind, we can express such polynomials in the basis
\[
 \Gamma_d = \{ t^i (1+t)^{d-2i} \}_{0\leq 2i \leq d}.
\]
That is, if $h$ is palindromic, there is a vector $\gamma = (\gamma_0, \gamma_1,\ldots, \gamma_{\lfloor d/2 \rfloor})$ such that
\[
 h(t) = \sum_{0 \leq 2i \leq d } \gamma_i t^i(1+t)^{d-2i}.
\]
Since each element of $\Gamma_d$ is unimodal and palindromic with the same center of symmetry, it follows that if $h$ has a nonnegative $\gamma$-vector, then its coefficients are also unimodal: $h_0 \leq h_1 \leq \cdots \leq h_{\lfloor d/2 \rfloor} \geq \cdots \geq h_{d-1} \geq h_d$. 

While the fact that the classical Eulerian polynomials are gamma-nonnegative dates back to the work of Foata and Sch\"utzenberger \cite[Theorem 5.2]{FS}, the more general notion of gamma-nonnegativity first arose in work of Gal \cite{Gal}, in the context of combinatorial topology. There is a growing literature on gamma-nonnegativity; introductions to the topic can be found in \cite{Branden2015}, \cite[Chapter 4]{Petersen2015}, or \cite{ath}.

A result of Br\"and\'en \cite[Theorem 4.2]{Br} (see also \cite[Theorem 6.1]{Br2} and \cite[Corollary 7.10]{Stembridge2007}) says that the $P$-Eulerian polynomial of any sign-graded poset of rank $r$ is gamma-nonnegative in the basis $\Gamma_{n-r-1}$. Actually, the statement is phrased in terms of the descent generating function without the extra factor of $t$, so in our notation it says that $A_P(t)/t$ is gamma-nonnegative in the basis $\Gamma_{n-r-1}$. In particular, naturally labeled graded posets are sign-graded. As the zig-zag poset $\varepsilon Z_n$ is naturally labeled and graded of rank 1, we can conclude the following result.

\begin{thm} \label{t-gamma}
For each $n \geq 1$, the polynomial $Z_n(t)/t$ has a nonnegative expansion in the basis $\Gamma_{n-2}$, i.e., there exist nonnegative integers $\gamma_{n,j}$ such that
\[
 Z_n(t) = t\cdot\sum_{0\leq 2j\leq n-2} \gamma_{n,j} t^j(1+t)^{n-2-2j} = \sum_{0\leq 2j\leq n-2} \gamma_{n,j}t^{j+1}(1+t)^{n-2(j+1)}.
\]
In particular, $Z_n(t)$ has degree $n-1$, palindromic degree $n$, and the coefficents $\{z(n,k)\}_{0\leq k \leq n-2}$ are symmetric and unimodal.
\end{thm}

For example, $Z_5(t) = t(1+t)^3+4t^2(1+t)$ so $\gamma_{5,0} = 1$ and $\gamma_{5,1} =4$. The $\gamma$-vectors for $n\leq 10$ are displayed in Table \ref{tab:gamma}.

\begin{table}
\[
 \begin{array}{r|ccccccccc}
  n\backslash k & 0 & 1 & 2 & 3 & 4  \\
  \hline 
  1 & 1 \\
  2 & 1 \\
  3 & 1 &  \\
  4 & 1 & 1 & \\
  5 & 1 & 4 &  & \\
  6 & 1 & 10 & 5 &   \\
  7 & 1 & 21 & 36 &  \\
  8 & 1 & 40 & 159 & 45 \\
  9 & 1 & 72 & 556 & 528  \\
  10 & 1 & 125 & 1697 & 3612 & 665 
 \end{array}
\]
\caption{Table of values of $\gamma_{n,j}$ for $1\leq n \leq 10$ and $0\leq 2j \leq n-2$.}\label{tab:gamma}
\end{table}

\section{Alternating permutations and relaxed $P$-partitions} \label{s-altrel}

In this section, we show that the zig-zag Eulerian polynomials encode the distribution of a permutation statistic over alternating permutations.

\subsection{Big returns for alternating permutations}\label{sec:alt}

Recall that an \emph{up-down alternating permutation} is a permutation $u\in S_n$ such that $u(1)<u(2)>u(3)<\cdots$, i.e., $i$ is a descent if and only if $i$ is even. We denote by $U_n$ the set of all such permutations. Similarly, a \emph{down-up alternating permutation} is a permutation $v \in S_n$ such that $v(1) > v(2) < v(3) > \cdots$, and we let $\bar U_n$ denote the set of down-up alternating permutations. It is well known that $|U_n| = |\bar U_n| = E_n$, where $E_n$ denotes the $n$th Euler number. This sequence of cardinalities begins
\[
 1,1,2, 5, 16, 61, 272, 1385, \ldots.
\]
 
The set of alternating permutations is a rich and well-studied subset of $S_n$; see \cite{Stanley2010} for a survey of alternating permutations and some of their connections to other parts of mathematics. In particular, the following observation shows that the number of linear extensions of any labeling of $Z_n$ (resp.\ $\bar Z_n$) is equal to the number of up-down (resp.\ down-up) alternating permutations in $S_n$. 

\begin{obs}\label{obs:ZU}
For any $n \geq 1$ and permutation $\sigma$,
\begin{itemize}
\item $\pi \in \LL(\sigma Z_n)$ if and only if $\pi^{-1}\sigma \in U_n$, and
\item $\pi \in \LL(\sigma \bar Z_n)$ if and only if $\pi^{-1}\sigma \in \bar U_n$.
\end{itemize}
\end{obs}

\begin{proof}
We have $\pi \in \LL(\sigma Z_n)$ if and only if $\sigma(1)<_{\pi} \sigma(2) >_{\pi} \sigma(3) <_{\pi} \cdots$, i.e., 
\[
 \pi^{-1}(\sigma(1)) < \pi^{-1}(\sigma(2)) > \pi^{-1}(\sigma(3)) < \cdots .
\]
That is, $\pi \in \LL(\sigma Z_n)$ if and only if $\pi^{-1}\sigma \in U_n$. The same argument with all inequalities reversed shows that $\pi \in \LL(\sigma\bar Z_n)$ if and only if $\pi^{-1}\sigma \in \bar U_n$.
\end{proof}

As alternating permutations are so well-studied, it seems worthwhile to find a combinatorial interpretation for $Z_n(t)$ in terms of a permutation statistic for the set $U_n$.  As first observed by Coons and Sullivant \cite{CoonsSullivant2023}, one such statistic is the number of \emph{big returns} (they call these ``swaps''), as we will show following some definitions. 

First, we say that a descent $i$ of $w$ is a \emph{big descent} of $w$ if $w(i)>w(i+1)+1$. In fact, we can define---for any $r\geq 0$---an $r$-descent set and $r$-descent number as follows:
\[
 \Des_r(w) = \{\, i : w(i)>w(i+1)+r \,\} \quad \mbox{ and } \quad \des_r(w)=\left|\Des_r(w)\right|.
\] 
Note that big descents correspond to the $r=1$ case. The study of the statistic $\des_r$ over the full symmetric group has many similarities to the classical ($r=0$) Eulerian distribution; see Foata and Sch\"utzenberger \cite[Section II.6]{FS}.

A \emph{return} of a permutation $w$ is a descent of its inverse permutation $w^{-1}$. In other words, a return is a value $j$ such that $j$ appears to the right of $j+1$ when $w$ is written in one-line notation. Returns are also known as ``left descents'' or simply ``inverse descents'' in the literature. We extend this notion to $r$-returns via the following definition:
\[
 \Ret_r(w) = \Des_r(w^{-1}) \quad \mbox{ and } \quad \ret_r(w) = \left|\Ret_r(w)\right|.
\]
\emph{Big returns} correspond to the $r=1$ case and ordinary returns correspond to $r=0$, for which we simplify notation by writing $\Ret(w)=\Ret_0(w)$ and $\ret(w) = \ret_0(w)$. 

For example, $w=579842316$ has $\Des(w) = \{3,4,5,7\}$, $\Des_1(w) = \{ 4,5,7\}$, $\Ret(w) = \{1,3,4,6,8\}$, and $\Ret_1(w) = \{1,3,4,6\}$, so $\des(w) = 4$, $\des_1(w) =3$, $\ret(w) = 5$, and $\ret_1(w) = 4$.

\begin{table}
\[
\begin{array}{c | c | c | c}
 \ret_1(u)=0 & \ret_1(u)=1 & \ret_1(u)=2 & \ret_1(u)=3 \\
 \hline
 \hline
 13254 & 1425\mathbf{3} & 153\mathbf{4}\mathbf{2} & 35\mathbf{2}\mathbf{4}\mathbf{1} \\
       & 1435\mathbf{2} & 24\mathbf{1}5\mathbf{3} & \\
       & 152\mathbf{4}3 & 25\mathbf{1}\mathbf{4}3 & \\
       & 23\mathbf{1}54 & 253\mathbf{4}\mathbf{1} & \\
       & 2435\mathbf{1} & 34\mathbf{2}5\mathbf{1} & \\
       & 3415\mathbf{2} & 351\mathbf{4}\mathbf{2} & \\
       & 451\mathbf{3}2 & 452\mathbf{3}\mathbf{1} & 
\end{array}
\]
\caption{The alternating permutations in $U_5$, grouped according to the number of big returns, which are highlighted in bold.}\label{tab:altex}
\end{table}

Define
\[
 U_n(t) = \sum_{u \in U_n} t^{\ret_1(u)} = \sum_{k\geq 0} u(n,k) t^k,
\]
so that $u(n,k)$ is the number of alternating permutations with $k$ big returns, i.e.,
\[
 u(n,k) = \left|\{\, u \in U_n : \ret_1(u) = k \,\}\right|.
\]
In Table \ref{tab:altex} we see the alternating permutations in $U_5$, grouped according to the number of big returns, so $U_5(t) = 1+7t+7t^2+t^3$; compare with Table \ref{tab:zigex}. We will give two proofs of the following result in Section \ref{sec:proof1}. This result is essentially \cite[Theorem 13]{CoonsSullivant2023}, though our proofs are different.

\begin{thm}\label{thm:ZU}
The number of big returns over $U_n$ is distributed the same as the number of descents over linear extensions of $\varepsilon Z_n$. That is, for each $n\geq 1$, we have $u(n,k)=z(n,k)$ for all $k$, and
\[
 Z_n(t) = \sum_{w\in \LL(\varepsilon Z_n)} t^{1+\des(w)} = \sum_{u \in U_n} t^{1+\ret_1(u)} = tU_n(t).
\]
\end{thm}

By Theorem \ref{thm:ZbarZ}, we could have also phrased Theorem \ref{thm:ZU} in terms of down-up permutations and linear extensions of $o \bar Z_n$. 

\begin{rem}
From \cite[Theorem 6.4]{Br2} we can obtain a combinatorial description of the numbers $\gamma_{n,j}$ in terms of peaks of linear extensions in $\LL(\varepsilon Z_n)$; this translates to a statistic for permutations in $U_n$, albeit one that is not very nice.
\end{rem}

Coons and Sullivant \cite{CoonsSullivant2023} ask for a direct combinatorial explanation for the symmetry and unimodality of the numbers $u(n,k)$, and we reiterate this desire.

\subsection{Relaxed $P$-partitions} \label{s-relaxed}

To help understand Theorem \ref{thm:ZU} and to derive later results, we shall use a modification of the notion of a $P$-partition, which we now define.

\begin{defn}[Relaxed $P$-partition]
A \emph{relaxed $P$-partition} is an order-preserving map $g: P \to \NN$ such that for $i <_P j$:
\begin{itemize}
\item $g(i) \leq g(j)$ if $i \leq j+1$,
\item $g(i) < g(j)$ if $i > j+1$.
\end{itemize}
\end{defn}

In other words, every ordinary $P$-partition is a relaxed $P$-partition, but there are some relaxed $P$-partitions that are not ordinary $P$-partitions. Specifically, if there is an unnatural pair of the form $i+1<_P i$, we can have $g(i+1)=g(i)$ for a relaxed $P$-partition $g$, whereas all ordinary $P$-partitions $f$ would have $f(i+1)<f(i)$. We denote the set of relaxed $P$-partitions by $\A^1(P)$.

For example, if $n=3$ and $P$ is the poset with $1 <_P 2$ and $3 <_P 2$, its relaxed $P$-partitions $g$ must satisfy
\[
 g(1)\leq g(2) \geq g(3),
\]
so
\[
 \A^1(P) = \{\, (a_1, a_2, a_3)\in  \NN^3 : a_1 \leq a_2 \geq a_3 \,\}.
\]
This is in contrast with its ordinary $P$-partitions, which form a subset $\A(P) \subseteq \A^1(P)$ with $a_3<a_2$. We can write $\A^1(P)$ as the disjoint union
\[
 \A^1(P) = \{ a_1 \leq a_3 \leq a_2 \} \cup \{ a_3 < a_1 \leq a_2 \},
\]
and each of these smaller sets can viewed as the relaxed $P$-partitions for a chain:
\[
 \A^1(P) = \A^1(132) \cup \A^1(312).
\]

However, this example is misleading, since it is not always the case that relaxed $P$-partitions can be grouped according to linear extensions in a disjoint manner. For example, say $Q$ is the poset given by $1<_Q 3 >_Q 2$. Then we see that 
\[
 \A^1(Q) = \{\, (a_1, a_2, a_3)\in  \NN^3 : a_1 \leq a_3 \geq a_2 \,\}
\]
and this can be expressed as a union over linear extensions:
\[
 \A^1(Q) = \A^1(123) \cup \A^1(213).
\] 
But this union is not disjoint, as
\[
 \A^1(123) \cap \A^1(213) = \{\, (a_1, a_2, a_3)\in  \NN^3 : a_1 = a_2 \leq a_3 \,\}.
\]

In general, we can make the following observation. 

\begin{obs}\label{obs:refine}
Let $P$ be a poset and suppose $i < j$ are incomparable in $P$. Let $Q = P\cup \{ i<_Q j\}$ and $R= P\cup \{ j <_R i \}$ be the refinements of $P$ that can be obtained by making $i$ and $j$ comparable. 
We have
\[
 \A^1(P) = \A^1(Q) \cup \A^1(R)
\]
and
\[
\A^1(Q) \cap \A^1(R) = 
\begin{cases} \{\, g \in \A^1(P) : g(i)=g(i+1) \,\}, & \mbox{ if $j=i+1$},\\
 \emptyset, & \mbox{ otherwise}.
 \end{cases}
\]
\end{obs}

By induction on the number of incomparable pairs in a poset $P$, Observation \ref{obs:refine} implies that if all adjacent integers are comparable in $P$, the set of all relaxed $P$-partitions \emph{is} the disjoint union of the $P$-partitions for its chains. We say that a poset $P$ is \emph{successive} if $i$ is comparable to $i+1$ for each $i < n$. For example, the antichain is not successive for $n\geq 2$ while every chain (permutation) is successive. The poset $1 <_P 3 >_P 2$ (i.e., $\varepsilon Z_3$) is not successive, but $1 <_P 2 >_P 3$ (i.e., $Z_3$) is successive.

We collect this discussion in the following lemma, which is the relaxed version of Lemma~\ref{thm:ftpp}.

\begin{lemma}[Fundamental lemma of relaxed $P$-partitions]\label{thm:ftpp2}
Suppose $P$ is successive. Then the set of relaxed $P$-partitions is the disjoint union of the relaxed $\pi$-partitions of its linear extensions $\pi$:
\[
 \A^1(P) = \bigcup_{\pi \in \LL(P)} \A^1(\pi).
\]
\end{lemma}

By analogy with the ordinary case, let $\A^1(P;m)$ be the set of $P$-partitions bounded by $m$, i.e.,
\[
 \A^1(P;m) = \{\, g \in \A^1(P) : 1 \leq g(i) \leq m \text{ for all }i \,\},
\]
and define the \emph{relaxed order polynomial} of $P$ to be $\op^1_P(m) = \left|\A^1(P;m) \right|$.
An immediate corollary of Lemma \ref{thm:ftpp2} is that relaxed order polynomials of successive posets are sums of relaxed order polynomials of their linear extensions.

\begin{cor}\label{cor:oplinearex2}
The relaxed order polynomial of a successive poset $P$ is the sum of the relaxed order polynomials of its linear extensions:
\[
 \op^1_P(m) = \sum_{\pi \in \LL(P)} \op^1_{\pi}(m).
\]
\end{cor}

We can now give relatively simple expressions for the relaxed order polynomials of chains. A function $g \in \A^1(\pi)$ if and only if 
\[
 g(\pi(1)) \leq g(\pi(2)) \leq \cdots \leq g(\pi(n)),
\]
with $g(\pi(i)) < g(\pi(i+1))$ whenever $\pi(i) > \pi(i+1)+1$, i.e., when $i$ is a big descent of $\pi$. Analogously to Equation \eqref{eq:piop} for the (ordinary) order polynomials of chains, we deduce
\[
\op^1_{\pi}(m) = \binom{m+n-1-\des_1(\pi)}{n},
\]
and with Theorem \ref{thm:ftpp2} we obtain the following result; compare with Equation \eqref{eq:ogf}.

\begin{prop}\label{prp:decomp}
For any successive poset $P$ of $[n]$, we have
\[
 \sum_{m\geq 1} \op^1_P(m) t^m = \frac{\sum_{\pi \in \LL(P)} t^{1+\des_1(\pi)}}{(1-t)^{n+1}}. 
\]
\end{prop}

Continuing our analogous notation, we can write $A^1_P(t) = \sum_{\pi \in \LL(P)} t^{1+\des_1(\pi)}$ for the numerator in this relaxed order polynomial generating function, which we call the \emph{relaxed $P$-Eulerian polynomial}.

\begin{rem}
We can generalize the definition of relaxed $P$-partitions to ``$r$-relaxed'' $P$-partitions for any $r\geq 0$. There is a fundamental lemma for ``$r$-successive'' posets, and other analogous statements. The set of $r$-relaxed $P$-partitions will be denoted $\A^r(P)$ and their order polynomials $\op^r_P(m)$, hence the notations $\A^1(P)$ and $\op^1_P(m)$ for the relaxed ($r=1$) case. We have further comments about $r$-relaxed $P$-partitions and generalizations of zig-zag posets in Section \ref{sec:morerelaxed}.
\end{rem}

\subsection{Proof of Theorem \ref{thm:ZU}}\label{sec:proof1}

Now we consider relaxed $P$-partitions for the zig-zag poset $Z_n$. Let $Z_n^1(t) = A^1_{Z_n}(t)$ denote the relaxed $P$-Eulerian polynomial and $\op_n^1(m)=\op_{Z_n}^1(m)$ the relaxed order polynomial. This poset is quite clearly successive, and therefore its relaxed $P$-partition generating function is of the form
\[
 \frac{Z_n^1(t)}{(1-t)^{n+1}} = \sum_{m\geq 0} \op_n^1(m) t^m.
\]
However, by taking $\sigma$ to be the identity in Observation \ref{obs:ZU}, we see that $\LL(Z_n)$ is the set of inverse alternating permutations, i.e., $u \in U_n$ if and only if $u^{-1}=w \in \LL(Z_n)$. Recall that $\ret_1(u) = \des_1(w)$. Thus it follows easily that
\begin{equation}\label{eq:z'u}
 Z_n^1(t) = \sum_{w \in \LL(Z_n)} t^{1+\des_1(w)} = \sum_{u \in U_n} t^{1+\ret_1(u)}=tU_n(t).
\end{equation}

It remains to show that $Z_n^1(t) = Z_n(t)$, which we will prove in Theorem \ref{thm:ZU}.

The action of the permutation $\varepsilon$ that takes $Z_n$ to the natural labeling $\varepsilon Z_n$ acts on functions by $\varepsilon f(i) = f(\varepsilon(i))$, and the following proposition is almost immediate from the definitions.

\begin{prop}\label{prp:f'f}
For any $n\geq 1$, we have 
\[
 \A^1(Z_n) = \{\, \varepsilon f : f \in \A(\varepsilon Z_n) \,\}.
\]
\end{prop}

\begin{proof}
Every relaxed $P$-partition $g \in \A^1(Z_n)$ has the property that
\begin{equation}\label{eq:g}
 g(1) \leq g(2) \geq g(3) \leq \cdots \geq g(2i-1) \leq g(2i) \geq g(2i+1) \leq \cdots,
\end{equation}
whereas every ordinary $P$-partition $f \in \A(\varepsilon Z_n)$ satisfies
\begin{equation}\label{eq:f}
 f(1) \leq f(3) \geq f(2) \leq \cdots \geq f(2i-2) \leq f(2i+1) \geq f(2i) \leq \cdots.
\end{equation}
From the definition of $\varepsilon$ we can see that if $g$ satisfies \eqref{eq:g}, then $\varepsilon g$ satisfies \eqref{eq:f}. Conversely, if $f$ satisfies \eqref{eq:f}, then $\varepsilon f$ satisfies \eqref{eq:g}.
\end{proof}

\begin{cor}\label{cor:or}
For any $n\geq 1$ and $m\geq 1$, we have
\begin{equation}
 \op_n^1(m) = \op_n(m) \label{e-opud}
\end{equation}
and hence
\[
 tU_n(t)=Z_n^1(t)= Z_n(t).
\]
\end{cor}

\begin{proof}
Combine Proposition \ref{prp:decomp} with Proposition \ref{prp:f'f} and Equation \eqref{eq:z'u}.
\end{proof}

This finishes one proof of Theorem \ref{thm:ZU}, but we can also give a bijective proof of Theorem \ref{thm:ZU} with the following.

\begin{prop}\label{prp:bij}
The map $\LL(\varepsilon Z_n) \to U_n$ given by $\pi \mapsto u = \pi^{-1}\varepsilon$ is a bijection. Moreover, for $\pi \in \LL(\varepsilon Z_n)$, we have
\[
 \Des(\pi) = \Ret_1(u).
\]
\end{prop}

Since $\varepsilon$ is the permutation that labels $\varepsilon Z_n$ as we draw the elements from left to right, we have $\pi \in \LL(\varepsilon Z_n)$ if and only if 
\[
\pi^{-1}(\varepsilon(1)) < \pi^{-1}(\varepsilon(2)) > \pi^{-1}(\varepsilon(3)) < \cdots .
\]
Thus $\pi \in \LL(\varepsilon Z_n)$ if and only if $u = \pi^{-1}\varepsilon \in U_n$. By the definition of big descents and big returns, it suffices to show $\Des(\pi) = \Des_1(\pi')$, where $\pi' = u^{-1} = \varepsilon \pi \in \LL(Z_n)$.

Here is a key observation.

\begin{obs}\label{obs:BDes}
For any permutation $\sigma$, we have $\Des_1(\sigma) \subseteq \Des(\varepsilon \sigma)$. 
\end{obs}

\begin{proof}
Suppose $j \in \Des_1(\sigma)$ is a big descent, with $\sigma(j) \in \{2i,2i+1\}$. Then $2i-1 \geq \sigma(j+1)$. Thus if $\sigma'=\varepsilon \sigma$, we have $\sigma'(j) \in \{2i,2i+1\}$ and still $2i-1 \geq \sigma'(j+1)$. Therefore $j \in \Des(\sigma')$. 
\end{proof}

We remark that $\Des_1(\sigma)$ need not equal $\Des_1(\varepsilon \sigma)$. If $\sigma(j) =2i+1$ and $\sigma(j+1)=2i-2$, then $\sigma'(j) = 2i$ and $\sigma'(j+1) = 2i-1$. So in such a case, $j \in \Des_1(\sigma)$ but $j \notin \Des_1(\sigma')$. 

\begin{proof}[Proof of Proposition \ref{prp:bij}]
As $\pi = \varepsilon \pi'$, Observation \ref{obs:BDes} shows us that $\Des_1(\pi') \subseteq \Des(\pi)$. We now show the reverse inclusion. 

There are two cases for $j \in \Des(\pi)$; either $j\in \Des_0(\pi)-\Des_1(\pi)$ or $j\in \Des_1(\pi)$. 

First consider $j \in \Des_0(\pi)-\Des_1(\pi) = \{\, s : \pi(s) = \pi(s+1)+1 \,\}$. By construction of the poset $\varepsilon Z_n$, the only possible ``small descents'' of a linear extension $\pi$ are of the form $\pi(j)=2i > 2i-1 = \pi(j+1)$ for some $i$. But when we act by $\varepsilon$, this gives $\pi'(j) = 2i+1 > 2i-2 = \pi'(j+1)$, and therefore $j \in \Des_1(\pi')$.

Observation \ref{obs:BDes} shows us that $\Des_1(\pi) \subseteq \Des(\pi')$. We will next show that if $j \in \Des_1(\pi)$, then $\pi'(j)\neq \pi'(j+1)+1$, and hence $\Des_1(\pi) \subseteq \Des_1(\pi')$. By construction of the poset $Z_n$, we have $2i >_{Z_n} 2i-1$ for all $i$, so no even number may be followed by the number one less than it in any linear extension. Thus if $\pi' \in \LL(Z_n)$, then the ``small descents'' of $\pi'$ have an odd number followed by an even number. More precisely, $j \in \Des_0(\pi')-\Des_1(\pi')$ if there is an $i$ such that $2i+1 = \pi'(j) > \pi'(j+1) = 2i$. But then $2i=\pi(j) < \pi(j+1) = 2i+1$, and so $j \notin \Des(\pi)$. The contrapositive is that if $j \in \Des(\pi)$, then $j$ is a big descent of $\pi'$.
\end{proof}

\begin{rem}\label{rem:zbar}
The proof of Proposition \ref{prp:f'f} can be readily adapted to prove the analogous result 
\[
 \A^1(\bar{Z}_n) = \{\, o f : f \in \A(o \bar{Z}_n) \,\}
\]
for down-up zig-zag posets, which implies
\begin{equation}
 \op^1_{\bar Z_n}(m) = \op_{o \bar Z_n}(m) = \op_n(m), \label{e-opdu}
\end{equation}
where the second equality is Equation \eqref{e-opuddu}. In Section \ref{s-F}, we will use \eqref{e-opdu} in proving a recurrence formula for the order polynomials $\op_n(m)$.
\end{rem}

\subsection{The fixed $m$ generating function} \label{s-F}

Before closing Section \ref{s-altrel}, we discuss the generating functions
\[
 \quad F_m(y) = \sum_{n\geq 0} \op_n(m) y^n,
\]
with $m$ fixed. This is the ordinary generating function for the $m$th column of Table \ref{tab:opnk}. Recall from Corollary \ref{cor:or} that $\op_n(m) = \op^1_n(m)$, so all statements in this section apply whether we interpret the coefficients to be counting $P$-partitions of $\varepsilon Z_n$ or relaxed $P$-partitions of $Z_n$. We will use the latter interpretation from now on, unless otherwise indicated.

Since $\op_n(1) = 1$ for all $n\geq 0$, we have $F_1(y) = 1/(1-y)$. For larger values of $m$, the function $F_m(y)$ can be computed from the following recurrences.

\begin{thm}\label{thm:Frec}
For all $m\geq 2$, we have
\begin{equation}\label{eq:negFrec}
 F_m(y) = \frac{1}{F_{m-1}(-y)-y}
\end{equation}
and
\begin{equation}\label{eq:posFrec}
 F_m(y) = \frac{y+2F_{m-1}(y)}{2-y^2-yF_{m-1}(y)}.
\end{equation}
\end{thm}

For example, we find
\[
F_2(y) = \frac{1+y}{1-y-y^2} = 1+2y+3y^2 + 5y^3 + 8y^4 + 13y^5+\cdots,
\]
so that $\op_n(2)$ is the sequence of Fibonacci numbers, starting with $1,2,3,5,\ldots$.

A recurrence equivalent to Equation \eqref{eq:negFrec} was established by B\'ona and Ju \cite{BonaJu2006}, though \eqref{eq:posFrec} appears to be new. From \eqref{eq:negFrec}, one can also deduce a recurrence equivalent to one observed by Stanley \cite[p. 31]{Stanley1973}:
\[
F_m(y) = \frac{1}{-y+F_{m-1}(-y)} = \frac{1}{-y+ \frac{1}{F_{m-2}(y)+y}} = \frac{y+F_{m-2}(y)}{1-y^2-yF_{m-2}(y)}.
\]
Below is one of many properties about the functions $F_m(y)$ that can be deduced from these recurrences. 

\begin{cor}
The functions $F_m(y)$ are rational of the form $P_m(y)/Q_m(y)$, where $m=\deg Q_m = 1+\deg P_m$. In particular, $P_1=1$, $Q_1=1-y$, $P_2 = 1+y$, $Q_2 =1-y-y^2$, and for all $m\geq 1$ we have
\[
\left[\begin{array}{cc} 1 & y \\ -y & 1-y^2 \end{array} \right] \left[ \begin{array}{c} P_m(y) \\ Q_m(y) \end{array} \right] = \left[ \begin{array}{c} P_{m+2}(y) \\ Q_{m+2}(y) \end{array} \right].
\]
\end{cor}

The proof of Theorem \ref{thm:Frec} follows from recurrences for order polynomials. As boundary cases, we have defined $\op_0(m) = 1$ and it is easily seen that $\op_1(m)=m$, while $\op_n(1)=1$ for all $n$.

\begin{prop}\label{prp:omrec}
For any $n\geq 2$ and $m\geq 1$, we have
\begin{align}
 \op_n(m) &= \op_n(m-1) + \op_{n-2}(m) + \sum_{0\leq j \leq (n-2)/2} \op_{2j+1}(m-1)\op_{n-2-2j}(m) \label{oprec1}\\
  &= \op_n(m-1) + \sum_{0\leq j \leq (n-1)/2} \op_{2j}(m-1)\op_{n-1-2j}(m). \label{oprec2} 
\end{align}
\end{prop}

Equation \eqref{oprec2} has been known since at least \cite[Example 2.3]{BermanKohler1976}, and OEIS entry A050446 takes this recurrence as its definition for the array of numbers $\op_n(m)$. As far as we know, \eqref{oprec1} is new. 

\begin{proof}[Proof of Proposition \ref{prp:omrec}]
Consider the zig-zag poset $Z_n$ and let $\A_n(m) = \A^1(Z_n;m)$. We will partition the set $\A_n(m)$ by cases. If $f \in \A_n(m)$ has $f(i)<m$ for all $i$, then $f \in \A_n(m-1)$, which accounts for the first term on the right-hand side of both \eqref{oprec1} and \eqref{oprec2}. From now on, we suppose some values of $f$ are equal to $m$.

Let $i$ be the minimal index with $f(i)=m$. We know either $i=1$ or $i$ is even, since  $f(2j)\geq f(2j+1)$ for all $j \geq 1$. If $i=1$, then $f(1)=m$, and since $f(1)\leq f(2)\leq m$, we have $f(2)=m$ as well. Then the remaining $n-2$ elements can be identified with any relaxed $P$-partition of $Z_{n-2}$ bounded by $m$, and these are counted by the $\op_{n-2}(m)$ term in \eqref{oprec1}.

Now suppose $i>1$, so that $i$ is even. We write $i=2j+2$ for some $0\leq j \leq (n-2)/2$. There are $2j+1$ elements to the left of $i$, and $f$ cannot exceed $m-1$ for these elements, thus there are $\op_{2j+1}(m-1)$ ways to assign values to $f(1),f(2),\ldots,f(2j+1)$. Meanwhile there are $n-1-(2j+1)= n-2-2j$ elements to the right of $i$, and these values can achieve $m$, so there are $\op_{n-2-2j}(m)$ ways to assign values to $f(2j+3),\ldots,f(n)$. The choices to the left and the right of $i$ are independent of each other, so there are a total of $\op_{2j+1}(m-1)\op_{n-2-2j}(m)$ elements of $\A_n(m)$ with $f(i)=m$. Summing over all $j$ gives the final term of \eqref{oprec1}.

Equation \eqref{e-opdu} in Remark \ref{rem:zbar} shows $\op_n(m) = \op^1_{\bar Z_n}(m)$, and so for the identity in \eqref{oprec2}, we will work with relaxed $P$-partitions $f$ for the down-up zig-zag posets $\bar Z_n$. That is, $\op_n(m)$ counts relaxed $P$-partitions with $f(1)\geq f(2)\leq f(3) \geq \cdots$. Let $i$ denote the minimal index such that $f(i)=m$, as before. Now it must be the case that $i=2j+1$ is odd. A similar argument as before tells us there are $\op_{2j}(m-1)$ ways to assign values to $f(1),f(2),\ldots,f(2j)$, and $\op_{n-1-2j}(m)$ ways to assign values to $f(2j+2),\ldots,f(n)$. Thus there are a total of $\op_{2j}(m-1)\op_{n-1-2j}(m)$ elements of $\A_n(m)$ with $f(i)=m$. Summing over all $j$ gives the final term of \eqref{oprec2}.
\end{proof}

\begin{rem} 
Proposition \ref{prp:omrec}, along with the identity
\[
\frac{Z_n(t)}{(1-t)^{n+1}} = \sum_{m\geq 0} \Omega_n(m) t^m,
\]
gives another efficient method for computing the zig-zag Eulerian polynomials $Z_n(t)$. First, we use either of the recurrences in Proposition \ref{prp:omrec} to compute the terms $\Omega_n(1), \Omega_n(2), \dots , \Omega_n(n)$, which allows us to form the polynomial
\begin{equation}
\Big( \sum_{m=0}^{n} \Omega_n(m) t^m \Big)(1-t)^{n+1}. \label{e-poly}
\end{equation}
Since $Z_n(t)$ has degree $n-1$, we obtain $Z_n(t)$ by truncating \eqref{e-poly} to a polynomial of degree $n-1$.
\end{rem}

We can now give a proof of Theorem \ref{thm:Frec}.

\begin{proof}[Proof of Theorem \ref{thm:Frec}]
Recall that $F_m(y) = 1 + my + \sum_{n\geq 2} \op_n(m) y^n$.
Adding \eqref{oprec1} and \eqref{oprec2} gives us
\begin{equation}\label{eq:double}
 2\op_n(m) = 2\op_n(m-1) + \op_{n-2}(m) + \sum_{0\leq j \leq n-1} \op_j(m-1)\op_{n-1-j}(m),
\end{equation}
while subtracting them gives us 
\begin{equation}\label{eq:alt}
\op_{n-2}(m) = \sum_{0\leq j\leq n-1} (-1)^j\op_j(m-1)\op_{n-1-j}(m).
\end{equation}

Taking \eqref{eq:double}, multiplying by $y^n$ on both sides and summing over all $n\geq 2$ yields
 \[
  2(F_m(y)-my-1) = 2(F_{m-1}(y)-(m-1)y-1) + y^2F_m(y) +y(F_{m-1}(y)F_m(y)- 1),
 \]
which simplifies to
\[
 (2-y^2-yF_{m-1}(y))F_m(y) = y+2F_{m-1}(y),
\]
and dividing by $2-y^2-yF_{m-1}(y)$ yields \eqref{eq:posFrec}.

Taking \eqref{eq:alt}, multiplying by $y^{n-1}$, and summing over all $n\geq 2$ yields
\[
 yF_m(y) = F_{m-1}(-y)F_m(y)-1;
\]
solving for $F_m(y)$ yields \eqref{eq:negFrec}.
\end{proof}

\section{Refinements and recursive formulas}\label{sec:refinement}

The purpose of this section is to introduce a refined order polynomial as a means to find recursive techniques for computing the zig-zag Eulerian polynomials.

The proof of Proposition \ref{prp:f'f} gives a bijection between the set of ordinary $P$-partitions $\A(\varepsilon Z_n)$ and the set of relaxed $P$-partitions $\A^1(Z_n)$. So to make our notation cleaner, as in Section \ref{s-F}, we will abuse notation and write $\A_n(m) = \A^1(Z_n;m)$ to indicate the set of bounded relaxed $P$-partitions of $Z_n$, and $\op_n(m) = \left|\A^1(Z_n;m)\right|$ for the relaxed order polynomial of $Z_n$. Unless otherwise stated, when we refer to a ``$P$-partition'' we mean a relaxed $P$-partition.

We define a $(p,q)$-analogue of $\op_n(m)$ which keeps track of the value $f(1)$ of each $P$-partition $f\in \A_n(m)$. Let 
\[
\op_n(p,q;m) = \sum_{f \in \A_n(m)} p^{f(1)}q^{m+1-f(1)} = \sum_{j=1}^m \op_{n,j}(m) p^jq^{m+1-j},
\]
so that $\op_{n,j}(m)$ is the number of $P$-partitions $f\in \A_n(m)$ satisfying $f(1) = j$. In some sense, this is only a one-parameter refinement since $\op_n(p,q;m) = q^{m+1}\sum_{f \in \A_n(m)} (p/q)^{f(1)}$, but the use of both $p$ and $q$ homogenizes the refinement and allows for simpler expressions later on.

For convenience, we define $\op_0(p,q;m) = pq^m$. For $n=1$, we have 
\[
\op_1(p,q;m) = \sum_{j=1}^m p^{j}q^{m+1-j} = pq\left(\frac{p^m-q^m}{p-q}\right).
\]
As another small example, with $n=2$, we see that $\A_2(m) = \{ 1\leq f(1) \leq f(2) \leq m \}$. Therefore, with each choice of $f(1)=j$, the value of $f(2)$ can be any of the values in $\{j,j+1,\ldots,m\}$. Thus, we have
\[
 \op_2(p,q;m) = \sum_{f \in \A_2(m)} p^{f(1)}q^{m+1-f(1)} = \sum_{j=1}^m (m+1-j) p^jq^{m+1-j}.
\]

We define
\[
 G_n(p,q,x) = \sum_{m\geq 1} \op_n(p,q;m) x^m,
\]
with $G_0(p,q,x) = pqx/(1-qx)$. By construction we have $\op_n(1,1;m) = \op_n(m)$, and we define $G_n(x) = G_n(1,1,x) = Z_n(x)/(1-x)^{n+1}$.

There is a fairly straightforward pattern as we move from the $P$-partitions of $Z_n$ to those of $Z_{n+1}$, which allows us to prove the following recurrence.

\begin{prop}\label{prp:Oqrec}
For all $n\geq 0$ and $m\geq 1$, we have
\begin{equation}\label{eq:Oqrec}
\op_{n+1}(p,q;m) = \frac{p}{p-q}\left[ \op_{n}(q,p;m) - \op_{n}(q,q;m) \right].
\end{equation}
\end{prop}

Multiplying both sides of \eqref{eq:Oqrec} by $x^m$ and summing over all $m\geq 1$, we have
\begin{equation}\label{eq:Gqrec}
G_{n+1}(p,q,x) = \frac{p}{p-q}\left[ G_n(q,p,x) - G_n(q,q,x) \right],
\end{equation}
which proves Theorem \ref{thm:refined}(a).  

\begin{proof}
Equation \eqref{eq:Oqrec} is straightforward to verify when $n=0$. Now suppose $n>0$. 

Notice that if $f \in \A(Z_{n+1};m)$, then $(f(2),\ldots,f(n+1)) \in \A(\overline{Z}_n;m)$ and $g=(m+1-f(2),\ldots,m+1-f(n+1)) \in \A(Z_n;m)$. Therefore, each $g\in \A(Z_n;m)$ with $g(1) = m+1-f(2)=j$ can be extended to $f\in \A(Z_{n+1};m)$ by choosing $f(1) \in \{1,2,\ldots,m+1-j\}$. Thus,
\begin{align*}
\op_{n+1}(p,q;m) &= \sum_{j=1}^m (pq^m+p^2q^{m-1}+\cdots+p^{m+1-j}q^j)\op_{n,j}(m)\\
 &= \sum_{j=1}^m pq^j\left(\frac{p^{m+1-j}-q^{m+1-j}}{p-q}\right)\op_{n,j}(m)\\
 &=\frac{p}{p-q}\sum_{j=1}^m q^jp^{m+1-j}\op_{n,j}(m) - \frac{p}{p-q}\sum_{j=1}^m q^{m+1} \op_{n,j}(m),
\end{align*}
which establishes the identity.
\end{proof}

From Equation \eqref{eq:Gqrec} we can compute the first few values of $G_n(q,x)$ recursively:
\begin{align*}
G_1(p,q,x) &=\frac{pqx}{(1-px)(1-qx)}, \\
G_2(p,q,x) &=\frac{pqx}{(1-px)(1-qx)^2}, \\
G_3(p,q,x) &=\frac{pqx(1-pqx^2)}{(1-px)^2(1-qx)^3},\\
G_4(p,q,x) &=\frac{pqx(1+qx-4pqx^2+p^2qx^3+p^2q^2x^4)}{(1-px)^3(1-qx)^4},\\
G_5(p,q,x) &=\frac{\scriptstyle pqx(1+px+3qx-12pqx^2+q^2x^2+4p^2qx^3-4pq^2x^3-p^3qx^4+12p^2q^2x^4-3p^3q^2x^5 -p^2q^3x^5-p^3q^3x^6)}{(1-px)^4(1-qx)^5}.
\end{align*}

With $n=5$, we have
\begin{multline*}
 G_5(p,q,x) = pqx+pq(5p+8q)x^2 + pq(14p^2+25pq+31q^2)x^3 + \\ pq(30p^3+56p^2q+75pq^2+85q^3)x^4 + \cdots
\end{multline*}
which refines \eqref{eq:G5}, and indeed by setting $p=q=1$ and simplifying we find
\[
 G_5(1,1,x) = \frac{x+7x^2+7x^3+x^4}{(1-x)^6}.
\]

\subsection{The refinement in terms of linear extensions}\label{sec:linearrefined}

While the recurrence in Theorem~\ref{thm:refined}(a) is good for computations, its connection to alternating permutations is rather opaque. We make the connection clearer in this section. To do this, we now count big returns of $u \in U_n$ according to whether they are greater or smaller than $u(1)$. First, define 
\[
\ret_1^-(u) = \left|\Ret_1(u) \cap \{1,2,\ldots,u(1)\}\right| \quad\text{and}\quad \ret_1^+(u) = \left|\Ret_1(u) \cap \{u(1)+1,\ldots,n-1\}\right|.
\]
Now we define 
\[
 U_n(s,t) = \sum_{u \in U_n} s^{\ret_1^-(u)}t^{\ret_1^+(u)}\left( \frac{1-t}{1-s}\right)^{u(1)},
\]
so that
\[
 U_n(t,t) = U_n(t).
\]
For example, with $n=4$, by summing over $u \in U_4 = \{ 1324, 1423, 2314, 2413, 3412\}$ we have
\[
 U_4(s,t) = \frac{1-t}{1-s} + \frac{t(1-t)}{(1-s)} + \frac{s(1-t)^2}{(1-s)^2}+ \frac{st(1-t)^2}{(1-s)^2} + \frac{s(1-t)^3}{(1-s)^3}.
\]

The main result of this section is the following theorem, from which it becomes much clearer why $G_n(1,1,x) = xU_n(x)/(1-x)^{n+1}$.

\begin{thm}\label{thm:Gperms}
For all $n\geq 1$, we have
\begin{equation}\label{eq:Gperms}
 G_n(p,q,x) = \sum_{u \in U_n} pqx\frac{ (px)^{\ret_1^-(u)}(qx)^{\ret_1^+(u)} }{(1-px)^{u(1)}(1-qx)^{n+1-u(1)}} = \frac{pqx U_n(px,qx) }{(1-qx)^{n+1}}.
\end{equation}
\end{thm}

For example, we can now express $G_4(p,q,x)$ as follows:
\begin{align*}
G_4(p,q,x) &= \frac{ pqx }{(1-px)(1-qx)^4} + \frac{ pq^2x^2 }{(1-px)(1-qx)^4} \\
 & \qquad + \frac{ p^2qx^2 }{(1-px)^2(1-qx)^3} + \frac{ p^2q^2x^3 }{(1-px)^2(1-qx)^3} + \frac{ p^2qx^2 }{(1-px)^3(1-qx)^2},\\
 &=\frac{pqx( \alpha(1+qx) + \alpha^2px(1+qx) + \alpha^3 px )}{(1-qx)^5},
\end{align*}
where $\alpha = (1-qx)/(1-px)$.

\begin{proof}[Proof of Theorem \ref{thm:Gperms}]
To prove Theorem \ref{thm:Gperms}, we present refined generating functions for linear extensions $\pi$ in $\LL(Z_n)$, phrased in terms of alternating permutations $u = \pi^{-1}$. 

By Theorem \ref{thm:ftpp2}, each $f \in \A_n(m)$ has $f \in \A^1(\pi;m)$ for some linear extension $\pi$ in $\LL(Z_n)$. That is,
\[
 G_n(p,q,x) = \sum_{\pi \in \LL(Z_n)} G_{\pi}(p,q,x)
\]
where $G_{\pi}(p,q,x) = \sum_{m\geq 1} \op_{\pi}(p,q;m) x^m$ and
\[
\op_{\pi}(p,q;m) = \sum_{f \in \A^1(\pi;m)} p^{f(1)}q^{m+1-f(1)}.
\]
To put it another way, an element $f$ of $\A^1(\pi;m)$ satisfies
\[
1\leq f(\pi(1))\leq f(\pi(2)) \leq \cdots \leq f(\pi(n)) \leq m
\]
with $f(\pi(i)) < f(\pi(i+1))$ if $i \in \Des_1(\pi)$, and to compute $G_{\pi}(p,q,x)$, we wish to assign the weight $p^{f(1)}q^{m+1-f(1)}x^m$ to each $f$. 

We accomplish the task with a variation on the notion of a \emph{barred permutation} on $\pi$, where we decorate $\pi$ with some number of bars between consecutive entries (or at the beginning or end) of $\pi$. (This is also known as the ``balls in boxes'' method as described in \cite[Section 1.9]{EC1}.) For $f \in \A^1(\pi;m)$, we place $m$ bars in total in $\pi$. Further, we require that there be at least one bar between $\pi(i)$ and $\pi(i+1)$ if $i \in \Des_1(\pi)$. The correspondence with $\pi$-partitions is that $f(a)$ is the number of bars to the left of $a$. To align our bookkeeping to ensure $1\leq f(\pi(1))$, we also require that there is at least one bar to the left of $\pi(1)$. In this scheme, we are then assigning $f$ the weight
\[
 p^{i}q^{j+1}x^{i+j}=p^{\scriptstyle (\# \textrm{ bars left of } 1)}q^{(1+\scriptstyle \# \textrm{ bars right of } 1)}x^{\scriptstyle (\# \textrm{ of bars} )}.
\]

For example, suppose $u = 382719465 \in U_9$, so that $\pi = u^{-1} = 531798426 \in \LL(Z_9)$. An example of a barred permutation on $\pi$ is
\[
 \pi'=|5|3||1|798|4|2|6||| ,
\]
which would correspond to the $\pi$-partition
\[
 (f(5), f(3), f(1), f(7), f(9), f(8), f(4), f(2), f(6)) = (1,2,4,5,5,5,6,7,8).
\]
Here there are eleven bars total, and we say $\pi'$ has weight $\wt(\pi') = p^4 q^8 x^{11}$. 

If $u(1) = r$, then $\pi(r)=1$, and hence there are $r$ gaps to the left of $1$ in which bars can be placed independently. Each of these gaps contributes a factor of $1/(1-px)=1+px+p^2x^2+\cdots$ to the generating function for all barred permutations on $\pi$. Likewise, there are $n+1-r$ gaps to the right of $1$, and these each contribute a factor of $1/(1-qx)$. In the gaps with big descents, and to the left of $\pi(1)$,  we must have at least one bar. If $i$ is the number of big descents among the first $r$ gaps, and $j$ is the number of big descents among the final $n+1-r$ gaps, we see
\[
 G_{\pi}(p,q,x) = \sum_{\scriptstyle \textrm{barred permutations $\pi'$ on $\pi$}} \wt(\pi') = \frac{(px)^{i+1}}{(1-px)^r}\cdot \frac{q(qx)^j}{(1-qx)^{n+1-r}}.
\]

We translate this expression into the language of alternating permutations by recalling that
\[
 \ret_1^-(u) = \left|\Ret_1(u) \cap \{1,2,\ldots,u(1)\}\right| = \left|\Des_1(\pi) \cap\{1,2,\ldots,u(1)\}\right|
\]
and
\[
 \ret_1^+(u) = \ret_1(u) - \ret_1^-(u).
\]
Thus, with $u = \pi^{-1}$, we have
\[
 G_{\pi}(p,q,x) = pqx\frac{ (px)^{\ret_1^-(u)} (qx)^{\ret_1^+(u)} }{ (1-px)^{u(1)}(1-qx)^{n+1-u(1)}}
\]
and the theorem follows.
\end{proof}

As a corollary of Theorem \ref{thm:Gperms} and Theorem \ref{thm:refined}(a), we obtain the expression
\begin{equation}\label{eq:Unrec}
 U_{n+1}(s,t) = \frac{1-t}{s-t}\left[ \left(\frac{1-t}{1-s}\right)^{n+1} sU_n(t,s) - tU_n(t,t) \right]
\end{equation}
for any $n\geq 0$; this is Theorem \ref{thm:refined}(b).

\begin{proof}[Proof of Theorem \ref{thm:refined}\textup{(}b\textup{)}]
By direct substitution of \eqref{eq:Gperms} in the recurrence given in \eqref{eq:Gqrec}, we have
\[
 \frac{pqxU_{n+1}(px,qx)}{(1-qx)^{n+2}} = \frac{p}{p-q}\left[\frac{pqx U_n(qx,px)}{(1-px)^{n+1}} - \frac{q^2xU_n(qx,qx)}{(1-qx)^{n+1}} \right].
\]
Canceling the like terms of $pqx$ and multiplying by $(1-qx)^{n+2}$ on both sides, we find
\[
U_{n+1}(px,qx) = \frac{1-qx}{p-q}\left[ \left( \frac{1-qx}{1-px} \right)^{n+1} pU_n(qx,px) - qU_n(qx,qx) \right],
\]
from which Equation \eqref{eq:Unrec} follows upon setting $p=s$, $q=t$, and $x=1$. 
\end{proof}

By taking appropriate limits, we can now prove
\begin{equation}\label{eq:Gdiff}
 G_{n+1}(x) = \frac{d}{dq}\left[ G_n(1,q,x) \right]_{q=1}
\end{equation}
and 
\begin{equation}\label{eq:Udiff}
 U_{n+1}(t) = t(1-t)\frac{d}{dt}\left[ U_n(s,t) \right]_{s=t} + (1+nt)U_n(t),
\end{equation}
which are parts (a) and (b1) of Corollary \ref{cor:deriv}, respectively. It is readily checked that (b1) and (b2) are equivalent given $Z(t)=tU_n(t)$.

\begin{proof}[Proof of Corollary \ref{cor:deriv}]
By setting $q=1$ in Equation \eqref{eq:Gqrec}, we have
\begin{equation}\label{eq:Grecq1}
G_{n+1}(p,1,x) = p\left[ \frac{G_n(1,p,x) - G_n(1,1,x)}{p-1}\right].
\end{equation}
Consider the limit as $p\to 1$. Taking the limit on the left side of \eqref{eq:Grecq1} gives $G_{n+1}(1,1,x) = G_{n+1}(x)$. Taking the limit on the right side yields $1\cdot \frac{d}{dp}\left[G_n(1,p,x)\right]_{p=1}$, which is equivalent to \eqref{eq:Gdiff}.

Next, we take the limit as $s\to t$ in \eqref{eq:Unrec}:
\begin{align*}
 U_{n+1}(t) &= \lim_{s\to t} U_{n+1}(s,t), \\
  &=(1-t)\lim_{s\to t}\left[ \frac{\left(\frac{1-t}{1-s}\right)^{n+1}sU_n(t,s) - \left(\frac{1-t}{1-t}\right)^{n+1}tU_n(t,t)}{s-t}\right], \\
  &=(1-t)^{n+2}\frac{d}{ds}\left[ \frac{sU_n(t,s)}{(1-s)^{n+1}} \right]_{s=t},\\
  &=(1-t)^{n+2}\left[\frac{s(1-s)\frac{d}{ds}\left[ U_n(t,s)\right] + (1+ns)U_n(t,s)}{(1-s)^{n+2}}\right]_{s=t},\\
  &=t(1-t)\frac{d}{ds}\left[U_n(t,s)\right]_{s=t} + (1+nt)U_n(t,t).
\end{align*}
Equation \eqref{eq:Udiff} follows by recognizing that $\frac{d}{ds}\left[ U_n(t,s) \right]_{s=t} = \frac{d}{dt}\left[ U_n(s,t)\right]_{s=t}$ and $U_n(t,t) = U_n(t)$.
\end{proof}

\section{Other interpretations for \texorpdfstring{$Z_n(t)$}{Z_n(t)} and \texorpdfstring{$\op_n(m)$}{Ω_n(m)}} \label{s-interp}

We now provide several interpretations for the number of $P$-partitions in $\A_n(m)=\A(Z_n;m)$ that we have found in the literature. A similar summary was provided by Xin and Zhong in \cite{XinZhong2022}.

\begin{itemize}
\setlength\itemsep{1em}

\item For any naturally labeled poset $P$ there is a convex simplicial cone in $\RR^n/\langle 1,1,\ldots,1\rangle$, called the \emph{Coxeter cone} $C(P)$, given by $x_i \leq x_j$ whenever $i <_P j$. The number of integer points in $C(P)\cap [0,m]^n$ is given by $\op_P(m+1)$. Moreover, $C(P)$ is a union of cones of the form $C(\pi)$, given by $x_{\pi(1)} \leq  x_{\pi(2)} \leq \cdots \leq x_{\pi(n)}$, where $\pi \in \LL(P)$. By intersecting $C(P)=\bigcup_{\pi \in \LL(P)} C(\pi)$ with a sphere, there is an induced simplicial complex $\Delta(P)$ of dimension $d=n-2$. The \emph{$f$-polynomial} of $\Delta(P)$ is the generating function for simplices in $\Delta(P)$ according to dimension, $f(t) = \sum_{F\in \Delta(P)} t^{\dim F}$, and the \emph{$h$-polynomial} of $\Delta(P)$ is the transformation of $f$ given by $h(t) = (1-t)^{d+1} f(t/(1-t))$. In the case of Coxeter cones of naturally labeled posets, we have $h(t) = A_P(t)/t$, and therefore $Z_n(t)/t$ is the $h$-polynomial for the simplicial complex $\Delta(\varepsilon Z_n)$. 

For example, with $n=4$, we have the following image of $\Delta(Z_4)$, which is the set $C(Z_4) =\{\, (x_1,x_2,x_3,x_4) : x_1\leq x_3 \geq x_2 \leq x_4 \,\}$, drawn in $\RR^4/\langle 1,1,1,1\rangle \cong \RR^3$, then intersected with a $2$-sphere:

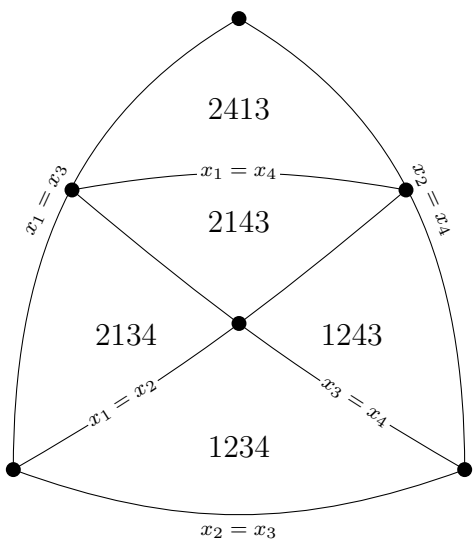
\[
\begin{tikzpicture}[xscale=.6, yscale=.6]
\tikzstyle{state1}=[rectangle,scale=1];
\draw (0,-4.5) node[state1] (1324)
 {$1234$};
\draw (-2.5,-2) node[state1] (2314)
 {$2134$};
\draw (2.5,-2) node[state1] (1423)
 {$1243$};
\draw (0,.5) node[state1] (2413)
 {$2143$};
\draw (0,3) node[state1] (3412)
 {$2413$};
\coordinate (a) at (0,5);
\coordinate (b) at (5,-5);
\coordinate (c) at (-5,-5);
\coordinate (d) at (3.7,1.2);
\coordinate (e) at (-3.7,1.2);
\draw (a) to [in=90,out=-30] node[midway,sloped,above] {\tiny $x_2=x_4$} (b);
\draw (b) to [in=-20,out=200] node[midway,sloped,below] {\tiny $x_2=x_3$} (c);
\draw (c) to [in=210,out=90] node[midway,sloped,above] {\tiny $x_1=x_3$} (a);
\draw (c) to [in=220,out=30] node[near start,sloped,inner sep=1,fill=white] {\tiny $x_1=x_2$} (d);
\draw (b) to [in=-40,out=150] node[near start,sloped,inner sep=1,fill=white] {\tiny $x_3=x_4$} (e);
\draw (d) to [in=10,out=170] node[midway,sloped,inner sep=1,fill=white] {\tiny $x_1=x_4$} (e);
\draw (a) node[circle, fill=black, inner sep=2] {};
\draw (b) node[circle, fill=black, inner sep=2] {};
\draw (c) node[circle, fill=black, inner sep=2] {};
\draw (d) node[circle, fill=black, inner sep=2] {};
\draw (e) node[circle, fill=black, inner sep=2] {};
\draw (0,-1.76) node[circle, fill=black, inner sep=2] {};
\end{tikzpicture} 
\]

We see the triangles here are labeled by linear extensions of $\varepsilon Z_4$. Counting faces, we have one empty face, six vertices, ten edges, and five triangles, so the $f$-polynomial is $f(t) = 1+6t+10t^2+5t^3$. By the transformation indicated, we have
\begin{align*}
h(t) &= (1-t)^3f(t/(1-t)) \\
 &= (1-t)^3+6t(1-t)^2 + 10t^2(1-t)+5t^3 \\
 &= 1+ (-3+6)t +(3-12+10)t^2+(-1+6-10+5)t^3 \\
 &= 1+3t+t^2,
\end{align*} 
which we recognize as $Z_4(t)/t$.

\item There are two polytopes commonly associated to any poset $P$, known as the \emph{order polytope} $\OO(P)$ (which is just $C(P)\cap [0,1]^n$) and the \emph{chain polytope} $\CC(P)$. The \emph{Ehrhart polynomial} of a polytope counts integer points in the $m$th dilation of the polytope. By \cite[Theorem 4.1]{Stanley1986}, for any finite poset $P$ the Ehrhart polynomials for $\OO(P)$ and $\CC(P)$ are the same, and they both equal $\op_P(m+1)$. In this paradigm, the \emph{$h^*$-polynomial} is the numerator of the generating function of $\op_P(m+1)$. Thus $h^*(\OO(Z_n);t)= h^*(\CC(Z_n);t)=Z_n(t)/t$. (While it turns out that the $h^*$-polynomial for these polytopes is the same as the $h$-polynomial of the simplicial complex previously mentioned, this is not true for all posets.) For $Z_n$, the order complex $\OO(Z_n)$ is defined by those points $\mathbf{x} \in \RR^n$ such that $0\leq x_i \leq 1$ for all $i\in[n]$ and
\[
x_1 \leq x_2 \geq x_3 \leq \cdots, 
\]
i.e., $x_{2i-1} \leq x_{2i}$ for $1\leq i \leq n/2$ and $x_{n-1}\geq x_n$ if $n$ is odd.
The chain complex $\CC(Z_n)$ is defined by points $\mathbf{y} \in \RR^n$ such that $0 \leq y_i \leq 1$ and
\[
y_i + y_{i+1} \leq 1 \mbox{ for } 1\leq i \leq n-1.
\]
As shown by Stanley \cite[Theorem 2.3]{Stanley1986} and reiterated by Coons and Sullivant \cite[Proposition 6.7]{CoonsSullivant2020}, these two polytopes are affinely equivalent under the invertible map $\phi \colon \RR^n \to \RR^n$ given by 
\[
\phi(x_i) = \begin{cases} 
 x_i, & \mbox{ if $i$ is odd},\\
 1-x_i, & \mbox{ if $i$ is even}.
\end{cases}
\]
It was conjectured by Kirillov \cite[Conjecture 3.11]{Kirillov2000} that the $h^*$-polynomial of $\CC(Z_n)$ has unimodal coefficients, and this was later proved independently by Chen and Zhang \cite{ChenZhang2016} and by Coons and Sullivant \cite{CoonsSullivant2023}; our gamma-nonnegativity result (Theorem \ref{t-gamma}) yields a third proof. 

In a different direction, we remark that Diaconis and Wood studied $\CC(Z_n)$ in \cite{DiaconisWood2013}, interpreting it as the set of all $(n+1)\times (n+1)$ \emph{tridiagonal doubly stochastic matrices}. For example, when $n=5$ these are the matrices:
\[
 \left( 
 \begin{array}{cccccc} 
 1-y_1 & y_1 & 0 & 0 & 0 & 0 \\ 
 y_1 & 1-y_1-y_2 & y_2 & 0 & 0 & 0 \\
 0 & y_2 & 1-y_2-y_3 & y_3 & 0 & 0 \\
 0 & 0 & y_3 & 1-y_3-y_4 & y_4 & 0 \\
 0 & 0 & 0 & y_4 & 1-y_4 -y_5 & y_5 \\
 0 & 0 & 0 & 0 & y_5 & 1- y_5 
 \end{array}
 \right)
\]
such that $y_i \geq 0$ for all $i$ and $1-y_i - y_{i+1} \geq 0$ for $i=1,\ldots,4$. While they do not do any explicit work with integer points in dilations of the polytope, they study the interplay between probabilistic results for these matrices and probabilistic results for random alternating permutations. 
 
\item To any graph we associate a \emph{magic labeling} by assigning nonnegative integer weights to the edges such that the sum of the weights at each vertex is a constant. (Loops are counted only once.) While Stanley studied the enumeration of magic labelings of graphs in \cite{Stanley1976}, it was in unpublished notes from 1973 \cite[pp.\ 30--31]{Stanley1973} that he considered, without proof, the magic labelings of a path graph on $n+1$ vertices with loops. For example, consider the graph below with $n=7$:
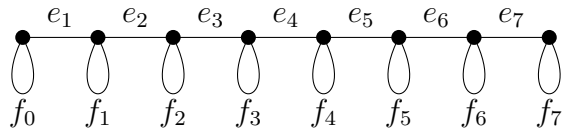
\[
 \begin{tikzpicture}
  \draw (0,0) node[circle,fill=black,inner sep=2] {} -- node[midway,above] {$e_1$} (1,0) node[circle,fill=black,inner sep=2] {} -- node[midway,above] {$e_2$} (2,0) node[circle,fill=black,inner sep=2] {} -- node[midway,above] {$e_3$} (3,0) node[circle,fill=black,inner sep=2] {}-- node[midway,above] {$e_4$} (4,0) node[circle,fill=black,inner sep=2] {} -- node[midway,above] {$e_5$} (5,0) node[circle,fill=black,inner sep=2] {}-- node[midway,above] {$e_6$} (6,0) node[circle,fill=black,inner sep=2] {}-- node[midway,above] {$e_7$} (7,0) node[circle,fill=black,inner sep=2] {};
  \foreach \x in {0,...,7}{
   \draw (\x,0) .. controls (\x-.5,-1) and (\x+.5,-1) .. (\x,0);
   \draw (\x,-1) node {$f_{\x}$};
  }
 \end{tikzpicture}
\]
For a fixed sum $m$, we give the loops weight $f_0 =m-e_1$, $f_n=m-e_n$, and $f_i = m-e_i - e_{i+1}$ for $1\leq i \leq n-1$. We have a magic labeling provided the $e_i$ satisfy $0\leq e_i \leq m$ and $e_i + e_{i+1} \leq m$ for each $1\leq i\leq n-1$. Thus the number of magic labelings with sum $m$ for the path graph on $n+1$ vertices is equal to the number of integer points in the $m$th dilation of the chain polytope $\CC(Z_n)$, which we know to be $\op_n(m+1)$.

This idea was explored in great detail in work of B\'ona and Ju \cite{BonaJu2006}, who provided many enumerative results for computing $\op_n(m)$. Both Stanley's 1973 notes and the paper of B\'ona and Ju are notable for including the computation of some small values of $Z_n(t)$. See \cite[Table 2]{BonaJu2006} for the coefficients of $Z_n(t)$ for $n\leq 9$. 

\item B\'ona, Ju and Yoshida generalized some of the perspective of \cite{BonaJu2006} in \cite{BonaJuYoshida2007}, where they counted integer points in \emph{edge polytopes} of graphs. These are defined by assigning nonnegative weights $y_i$ to vertices such that $y_i+y_j \leq 1$ if $i$ and $j$ are adjacent in the graph. Thus for the path graph on $n+1$ vertices, we have $y_i + y_{i+1}\leq 1$ for $1\leq i \leq n-1$ and the edge polytope is the same as the chain complex $\CC(Z_n)$. Thus, $\op_n(m+1)$ is its Ehrhart polynomial and $Z_n(t)/t$ is its $h^*$-polynomial. 

\item Gonz\'alez d'Le\'on, Hanusa, Morales, and Yip \cite{GHMY} showed that $\CC(Z_n)$ is integrally equivalent to the \emph{flow polytope} of the graph $G(2,n+2)$.  That is, $\op_n(m+1)$ is the Ehrhart polynomial of the $m$th dilation of $G(2,n+2)$. Without going into too much detail, $G(2,n+2)$ is a directed graph with vertex set $[n+2]$ and directed edges $i\to i+1$ for $1\leq i \leq n+1$, and $i \to i+2$ for $1\leq i\leq n$. We assign nonnegative weights to each directed edge, subject to the constraint that the sum of all incoming edge weights equal the sum of all outgoing edge weights. (This is known as the ``conservation of flow.'') Vertex 1 is a \emph{source} vertex, and vertex $n+2$ is the \emph{sink}. The flow polytope is the polytope in $\RR^{2n+1}$ whose coordinates are the edge weights. For example, $G(2,6)$ is shown below:
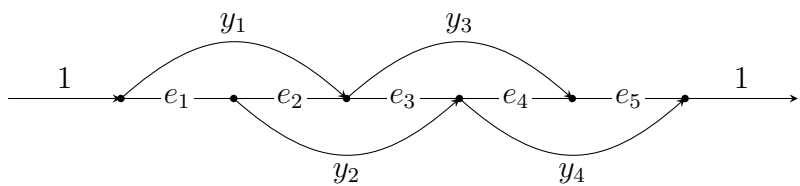
\[
\begin{tikzpicture}[>=stealth,xscale=1.5]
  \draw (0,0) node[circle,fill=black,inner sep=1] {} -- node[midway,fill=white,inner sep=1] {$e_1$} (1,0) node[circle,fill=black,inner sep=1] {} -- node[midway,fill=white,inner sep=1] {$e_2$} (2,0) node[circle,fill=black,inner sep=1] {} -- node[midway,fill=white,inner sep=1] {$e_3$} (3,0) node[circle,fill=black,inner sep=1] {}-- node[midway,fill=white,inner sep=1] {$e_4$} (4,0) node[circle,fill=black,inner sep=1] {} -- node[midway,fill=white,inner sep=1] {$e_5$} (5,0) node[circle,fill=black,inner sep=1] {};
  \foreach \x in {0,2}{
   \draw[->] (\x,0) .. controls (\x+.75,1) and (\x+1.25,1) .. (\x+2,0);
   \draw[->] (\x+1,0) .. controls (\x+1.75,-1) and (\x+2.25,-1) .. (\x+3,0);
  }
  \draw (1,1) node {$y_1$};
  \draw (2,-1) node {$y_2$};
  \draw (3,1) node {$y_3$};
  \draw (4,-1) node {$y_4$};
  \draw[->] (-1,0)-- node[midway,above] {$1$} (0,0);
  \draw[->] (5,0)-- node[midway,above] {$1$} (6,0);
 \end{tikzpicture}
\]
If the source has unit flow, we get the following system of equations:
\begin{align*}
 1 &= y_1+e_1, \\
 e_1 &=y_2+e_2, \\
 y_1+e_2 &= y_3+e_3, \\
 y_2+e_3 &= y_4+e_4, \\
 y_3+e_4 &= e_5, \\
 y_4+e_5 &=1.
\end{align*}
By adding consecutive equations, we can deduce $y_i+y_{i+1}+e_{i+1} = 1$, and therefore $y_i + y_{i+1}\leq 1$ as in the chain polytope. See also Brunner and Hanusa \cite{Brunner} for more recent work on flow polytopes of the graphs $G(2,n)$, which they call \emph{zigzag graphs}.

\item Coons and Sullivant \cite{CoonsSullivant2020} studied the \emph{toric vanishing ideal} for the Cavender--Farris--Neyman model for phylogenetic trees with a molecular clock. The associated polytope can be described in terms of an edge weighting of a rooted binary tree. Incredibly, any tree $T$ with $n+1$ leaves has Ehrhart polynomial $\op_n(m+1)$ \cite[Corollary 6.8]{CoonsSullivant2020}. When $T$ is the binary tree with $n+1$ leaves and $n$ internal nodes, this polytope has the same defining inequalities as the chain polytope $\CC(Z_n)$; see \cite[Corollary 4.13]{CoonsSullivant2020}.

\item Berman and K\"ohler \cite[Example 2.3]{BermanKohler1976} showed that $\op_n(m+1)$ is the number of order ideals of the poset $Z_n \times I_m$, where $I_m$ is an $m$-element chain. For each $i\in[n]$, we let $m+1-a_i$ denote the number of elements in the $i$th column of the ideal. Then $a_1\leq a_2 \geq a_3 \leq \cdots$. For example, with $n=7$ and $m=5$, here is a picture of the order ideal such that the vector of the number of elements per column is $(4,2,5,4,4,3,5)$, so that $(a_1,\ldots,a_7)=(1,3,0,1,1,2,0) \leftrightarrow (2,4,1,2,2,3,1) \in \A_7(6)$:
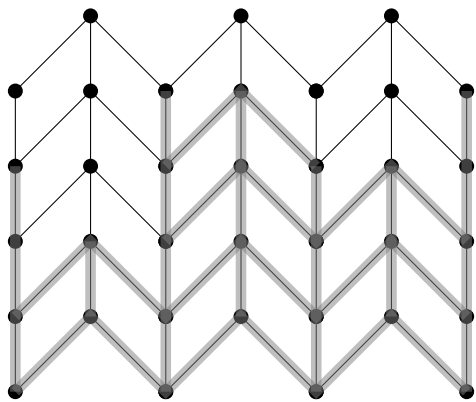
\[
\begin{tikzpicture}[baseline =3cm]
  \draw (0,1)--(0,5);
  \draw (1,2)--(1,6);
  \draw (2,1)--(2,5);
  \draw (3,2)--(3,6);
  \draw (4,1)--(4,5);
  \draw (5,2)--(5,6);
  \draw (6,1)--(6,5);
  \foreach \x in {1,3,5}{
   \foreach \y in {2,...,6}{
   \draw (\x-1,\y-1) node[circle,fill=black,inner sep=2] {} -- (\x,\y) node[circle,fill=black,inner sep=2] {} -- (\x+1,\y-1) node[circle,fill=black,inner sep=2] {};
   }
  }
  \draw[gray,opacity=.5,line width=4] (0,4)--(0,1);
  \draw[gray,opacity=.5,line width=4] (1,3)--(1,2);
  \draw[gray,opacity=.5,line width=4] (2,5)--(2,1);
  \draw[gray,opacity=.5,line width=4] (3,5)--(3,2);
  \draw[gray,opacity=.5,line width=4] (4,4)--(4,1);
  \draw[gray,opacity=.5,line width=4] (5,4)--(5,2);
  \draw[gray,opacity=.5,line width=4] (6,5)--(6,1);
  \draw[gray,opacity=.5,line width=4] (0,1)--(1,2)--(2,1)--(3,2)--(4,1)--(5,2)--(6,1);
  \draw[gray,opacity=.5,line width=4] (0,2)--(1,3)--(2,2)--(3,3)--(4,2)--(5,3)--(6,2);
  \draw[gray,opacity=.5,line width=4] (2,3)--(3,4)--(4,3)--(5,4)--(6,3);
  \draw[gray,opacity=.5,line width=4] (2,4)--(3,5)--(4,4);
 \end{tikzpicture}
 \]

\item In the chemistry literature, the study of \emph{Kekul\'e structures} of benzenoid hydrocarbons is modeled by perfect matchings of subgraphs of the hexagonal lattice. The major touchstone for this work is a book by Cyvin and Gutman \cite{CyvinGutman1988}. One family of graphs they considered is denoted $Z(n,m)$, which is an $n\times m$ grid of hexagons, best illustrated with an example. Below we see the grid $Z(7,5)$, along with one of its perfect matchings. 
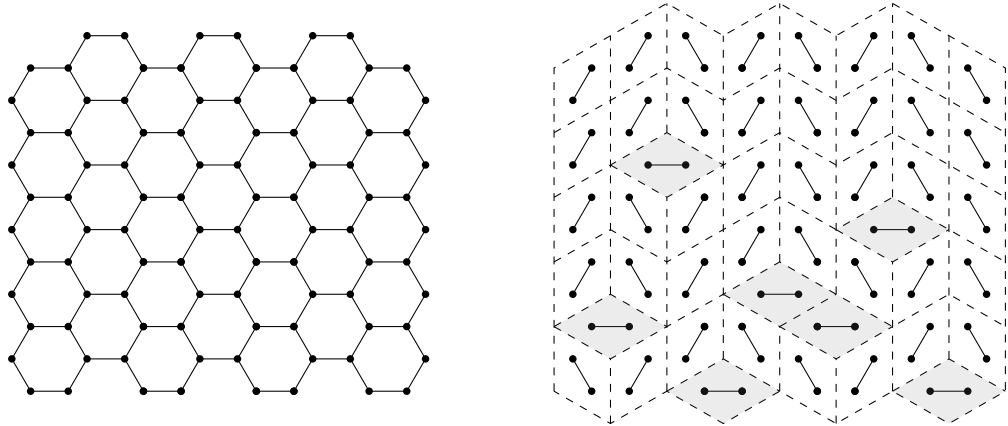
\[
\begin{tikzpicture}[xscale=.25,yscale=.43,baseline=3cm]
  \foreach \x in {3,9,15}{
   \foreach \y in {1,...,6}{
    \draw (\x-4,2*\y-1) node[circle,fill=black,inner sep=1] {};
    \draw (\x-2,2*\y-1) node[circle,fill=black,inner sep=1] {};
    \draw (\x-1,2*\y) node[circle,fill=black,inner sep=1] {};
    \draw (\x+1,2*\y) node[circle,fill=black,inner sep=1] {};
    \draw (\x+2,2*\y-1) node[circle,fill=black,inner sep=1] {};
    \draw (\x+4,2*\y-1) node[circle,fill=black,inner sep=1] {};
   }
  }
  \foreach \x in {-2,20}{
   \foreach \y in {1,...,5}{
    \draw (\x,2*\y) node[circle,fill=black,inner sep=1] {};
   }
  }
  \foreach \x in {0,...,3}{
   \foreach \y in {0,...,4}{
    \draw (6*\x-1,2*\y+3)--(6*\x-2,2*\y+2)--(6*\x-1,2*\y+1)--(6*\x+1,2*\y+1)--(6*\x+2,2*\y+2)--(6*\x+1,2*\y+3);
   }
  }
  \foreach \x in {0,1,2}{
   \foreach \y in {0,...,5}{
    \draw (6*\x+2,2*\y+2)--(6*\x+4,2*\y+2);
   }
  }
  \foreach \x in {0,...,3}{
   \draw (6*\x-1,11)--(6*\x+1,11);
  }
  \foreach \x in {0,1,2}{
   \draw (6*\x+1,11)--(6*\x+2,12);
   \draw (6*\x+4,12)--(6*\x+5,11);
  }
\end{tikzpicture}
\qquad \qquad 
\begin{tikzpicture}[xscale=.25,yscale=.43,baseline=3cm]
  \foreach \x in {3,9,15}{
   \foreach \y in {1,...,6}{
    \draw (\x-4,2*\y-1) node[circle,fill=black,inner sep=1] {};
    \draw (\x-2,2*\y-1) node[circle,fill=black,inner sep=1] {};
    \draw (\x-1,2*\y) node[circle,fill=black,inner sep=1] {};
    \draw (\x+1,2*\y) node[circle,fill=black,inner sep=1] {};
    \draw (\x+2,2*\y-1) node[circle,fill=black,inner sep=1] {};
    \draw (\x+4,2*\y-1) node[circle,fill=black,inner sep=1] {};
   }
  }
  \foreach \x in {-2,20}{
   \foreach \y in {1,...,5}{
    \draw (\x,2*\y) node[circle,fill=black,inner sep=1] {};
   }
  }
  \draw[dashed] (-3,1)--(-3,11)--(3,13)--(6,12)--(9,13)--(12,12)--(15,13)--(21,11)--(21,1)--(18,0)--(15,1)--(12,0)--(9,1)--(6,0)--(3,1)--(0,0)--(-3,1);
  \foreach \y in {7,9}{
   \draw[dashed] (-3,\y)--(3,\y+2)--(6,\y+1)--(9,\y+2)--(12,\y+1)--(15,\y+2)--(21,\y);
  }
  \draw[dashed] (0,4)--(0,12);
  \draw[dashed] (0,0)--(0,2);
  \draw[dashed] (3,9)--(3,13);
  \draw[dashed] (3,1)--(3,7);
  \draw[dashed] (6,2)--(6,12);
  \draw[dashed] (9,5)--(9,13);
  \draw[dashed] (9,1)--(9,3);
  \draw[dashed] (12,4)--(12,12);
  \draw[dashed] (12,0)--(12,2);
  \draw[dashed] (15,7)--(15,13);
  \draw[dashed] (15,1)--(15,5);
  \draw[dashed] (18,2)--(18,12);
  \draw[dashed] (-3,5)--(0,6)--(3,5)--(9,7)--(12,6)--(15,7)--(21,5);
  \draw[dashed] (-3,3)--(0,4)--(3,3)--(9,5)--(15,3)--(18,4)--(21,3);
  \draw[dashed] (-3,3)--(0,2)--(3,3);
  \draw[dashed] (0,8)--(3,7)--(6,8);
  \draw[dashed] (3,1)--(6,2)--(9,1);
  \draw[dashed] (6,4)--(12,2)--(15,3);
  \draw[dashed] (9,3)--(12,4);
  \draw[dashed] (12,6)--(15,5)--(18,6);
  \draw[dashed] (15,1)--(18,2)--(21,1);
  \draw (-1,3)--(1,3);
  \draw[fill=gray,opacity=.15,draw=none] (-3,3)--(0,4)--(3,3)--(0,2)--(-3,3);
  \foreach \y in {2,...,5}{
    \draw (-2, 2*\y)--(-1, 2*\y+1);
   }
  \draw (1,1)--(2,2);
  \draw (-2,2)--(-1,1);
  \draw (2,8)--(4,8);
  \draw[fill=gray,opacity=.15,draw=none] (0,8)--(3,9)--(6,8)--(3,7)--(0,8);
  \foreach \y in {5,6}{
    \draw (1, 2*\y-1)--(2, 2*\y);
    \draw (5, 2*\y-1)--(4, 2*\y);
  }
  \foreach \y in {2,3}{
    \draw (1, 2*\y+1)--(2, 2*\y);
  }
  \foreach \y in {1,2,3}{
    \draw (5, 2*\y+1)--(4, 2*\y);
  }
  \draw (5,1)--(7,1);
  \draw[fill=gray,opacity=.15,draw=none] (3,1)--(6,2)--(9,1)--(6,0)--(3,1);
  \draw (8,4)--(10,4);
  \draw[fill=gray,opacity=.15,draw=none] (6,4)--(9,5)--(12,4)--(9,3)--(6,4);
  \draw (7,3)--(8,2);
  \foreach \y in {3,...,6}{
    \draw (7, 2*\y-1)--(8, 2*\y);
    \draw (11, 2*\y-1)--(10, 2*\y);
  }
  \draw (11,3)--(13,3);
  \draw[fill=gray,opacity=.15,draw=none] (9,3)--(12,4)--(15,3)--(12,2)--(9,3);
  \draw (10,2)--(11,1);
  \draw (13,1)--(14,2);
  \draw (14,6)--(16,6);
  \draw[fill=gray,opacity=.15,draw=none] (12,6)--(15,7)--(18,6)--(15,5)--(12,6);
  \foreach \y in {4,...,6}{
    \draw (13, 2*\y-1)--(14, 2*\y);
    \draw (17, 2*\y-1)--(16, 2*\y);
  }
  \draw (13,5)--(14,4);
  \draw (17,5)--(16,4);
  \draw (17,3)--(16,2);
  \draw (17,1)--(19,1);
  \draw[fill=gray,opacity=.15,draw=none] (15,1)--(18,2)--(21,1)--(18,0)--(15,1);
  \foreach \y in {1,...,5}{
    \draw (20, 2*\y)--(19, 2*\y+1);
  }
\end{tikzpicture}
\]
As indicated with the shading and dashed lines, there is a clear correspondence with relaxed $P$-partitions of $Z_n$. Let $a_i$ be the height of the unique horizontal edge in column $i$. It can be shown that $0\leq a_i \leq m$ and $a_1 \leq a_2 \geq a_3 \leq \cdots$. Thus the number of perfect matchings of $Z(n,m)$---i.e., the number of Kekul\'e structures---is given by $\op_n(m+1)$. This matching would correspond to $(a_1,\ldots,a_7)=(1,3,0,1,1,2,0) \leftrightarrow (2,4,1,2,2,3,1) \in \A_7(6)$.

\end{itemize}

\section{Further thoughts and open questions} \label{s-conclusion}

We conclude our paper with a discussion of open questions and thoughts for future work.

\subsection{Real roots, log-concavity, and more}

The notion of gamma-nonnegativity is closely related to the properties of log-concavity and real-rootedness \cite{Branden2015}.  Furthermore, Nevo and Petersen \cite{NevoPetersen} have conjectured that many naturally occurring $\gamma$-vectors in fact satisfy the ``Frankl--F\"uredi--Kalai'' inequalities. We conjecture that all three properties hold for $Z_n(t)$.

\begin{conj} \label{cj-rr}
For all $n\geq 1$, each of the following is true:
\begin{enumerate}
 \item[(a)] The polynomial $Z_n(t)$ has only real roots.
 \item[(b)] The sequence $\{z(n,k)\}_{0\leq k\leq n-2}$ is log-concave, i.e., $z(n,k)^2 \geq z(n,k-1)z(n,k+1)$ for each $k\in\{1,2,\ldots,n-3\}$.
 \item[(c)] The vector $\gamma_n = (\gamma_{n,0},\gamma_{n,1},\ldots,\gamma_{n,\lfloor (n-2)/2 \rfloor })$ from Theorem \ref{t-gamma} satisfies the Frankl--F\"uredi--Kalai inequalities, i.e., $\gamma_n$ is the $f$-vector of a balanced simplicial complex. 
\end{enumerate}
\end{conj}

This conjecture has been verified for $n\leq 35$. In fact, we have observed that the roots of $(1-t)\frac{d}{dt}\left[ tU_n(s,t) \right]_{s=t}$ and $(n+1)tZ_n(t)$ appear to be real interlacing, which by Corollary \ref{cor:deriv}(b2) would imply part (a) of the conjecture.

\subsection{Entringer number refinements}

The \emph{Entringer number} $E_{n,r}$ is defined to be the number of alternating permutations $u \in U_n$ with $u(1) = r$ \cite{Entringer, Poupard}. Let 
$$U_{n,r} = \{\, u \in U_n : u(1) = r \,\},$$
so that $E_{n,r} = |U_{n,r}|$; see Table \ref{tab:Enr}.

\begin{table}
\[
  \begin{array}{r|cccccc}
  n\backslash r & 1 & 2 & 3 & 4 & 5 & 6
 \\
 \hline
  2& 1 \\
  3& 1 & 1 &  \\
  4& 2 & 2 & 1 \\
  5& 5 & 5 & 4 & 2 \\
  6& 16 & 16 & 14 & 10 & 5 \\
  7& 61 & 61 & 56 & 46 & 32 & 16
  \end{array}
\]
\caption{The Entringer numbers $E_{n,r}$ for $2 \leq n \leq 7$ and $1 \leq r \leq n-1$.}\label{tab:Enr}
\end{table}

We refine the Entringer numbers with the polynomials
\[
 U_{n,r}(s,t) = \sum_{u \in U_{n,r}} s^{\ret_1^-(u)}t^{\ret^+_1(u)},
\]
and we put these all together with
\[
 U_n(s,t,\alpha) = \sum_{u \in U_n} s^{\ret_1^-(u)}t^{\ret^+_1(u)}\alpha^{u(1)} = \sum_{r=1}^{n-1} U_{n,r}(s,t)\alpha^r.
\]
Notice that if $\alpha = (1-t)/(1-s)$, we have $U_n(s,t,(1-t)/(1-s))=U_n(s,t)$, as defined in Section \ref{sec:linearrefined}. For example,
\[
 U_4(s,t,\alpha) = (1+t)\alpha + s(1+t)\alpha^2 + s\alpha^3
\]
and
\[
 U_5(s,t,\alpha) = (1+3t+t^2)\alpha + s(2+3t)\alpha^2 + s(1+s)(1+t)\alpha^3 + s(1+s)\alpha^4.
\]

It can be proved bijectively that $E_{n,r} = E_{n-1,n-r}+E_{n,r+1}$ which, along with the boundary conditions $E_{n,r} = 0$ for $r<1$ and $r\geq n$, allows for quick recursive calculation of the numbers in Table \ref{tab:Enr}. We have not found a similarly nice linear recurrence for the polynomials $U_{n,r}(s,t)$, though we can make two easy observations about the end cases. 

\begin{obs}
For each $n\geq 2$, we have
\[
 U_{n+1,1}(s,t) = U_n(t) \qquad \text{and} \qquad U_{n+1,n}(s,t) = sU_{n-1}(s).
\]
\end{obs}

\begin{proof}
Suppose $u \in U_{n+1,1}$. Then $u = 1v$ where $v$ is a down-up permutation of  $\{2,\ldots,n+1\}$. Since $u(1) = 1$, all the big returns of $u$ are greater than $u(1)$, and thus $\ret_1(u) = \ret_1^+(u) = \ret_1(v)$. By Theorem \ref{thm:ZbarZ}, we know that the distribution of big returns for up-down alternating permutations is the same as for down-up permutations, and the first equation follows.

The reasoning for the second identity is similar. Now, we notice that if $u \in U_{n+1,n}$, then $u = n(n+1)v$ where $v \in U_{n-1}$. This time, every big return of $u$ is below $u(1) = n$, and there has to be a big return from $n-1$ to $n$. Therefore, $\ret_1(u) = \ret_1^-(u) = 1+\ret_1(v)$. Summing over all $v \in U_{n-1}$ gives the result.
\end{proof}

As a companion to Table \ref{tab:Enr}, our Table \ref{tab:Unr} lists the polynomials $U_{n,r}(s,t)$. Given the key role of $U_n(s,t)$ in this project, we believe the polynomials $U_{n,r}(s,t)$ deserve further study.

\begin{table}
\[
  \begin{array}{r|cccccc}
  n\backslash r & 1 & 2 & 3 
 \\
 \hline
  2& {\scriptstyle 1} \\
  3& {\scriptstyle 1} & {\scriptstyle s} &  \\
  4& {\scriptstyle 1+t} & {\scriptstyle s(1+t)} & {\scriptstyle s} \\
  5& {\scriptstyle 1+3t+t^2} & {\scriptstyle s(2+3t)} & {\scriptstyle s(1+s)(1+t)} \\
  6& {\scriptstyle 1+7t+7t^2+t^3 } & {\scriptstyle 2s(1+5t+2t^2)} & {\scriptstyle s(2+s+5t+4st+t^2+st^2)}  \\
  7& {\scriptstyle 1+14t+31t^2+14t^3+t^4} & {\scriptstyle s(3+26t+28t^2+4t^3)} & {\scriptstyle s(1+s)(3+17t+8t^2)} \\
  \hline 
  \\
  n\backslash r &  4 & 5 & 6
 \\
 \hline
  5&  {\scriptstyle s(1+s)}\\
  6&  {\scriptstyle s(1+t)(2+3s)} &  {\scriptstyle s(1+3s+s^2)}\\
  7&  {\scriptstyle s(3+6s+7t+s^2+19st+t^2+4st^2+4s^2t+s^2t^2)} 
 &  {\scriptstyle 2s(1+t)(1+5s+2s^2)} & {\scriptstyle s(1+7s+7s^2+s^3)}
  \end{array}
\]
\caption{The polynomials $U_{n,r}(s,t)$ for $2\leq n \leq 7$ and $1\leq r \leq n-1$.}\label{tab:Unr}
\end{table}

\subsection{Other combinatorial interpretations}

One of the main contributions of the present work is to give an avenue for combinatorial understanding of the zig-zag Eulerian polynomials in terms of permutation statistics, with the hope that it will ultimately lead to a better understanding of these polynomials and their calculation, recursive or otherwise. This endeavor was partly successful.

There are a variety of other combinatorial interpretations for the Euler numbers \cite{Stanley2010}, so it might be the case that another interpretation reveals more of the structure. Among the possibilities, we found two other sets of combinatorial objects along with statistics whose distributions appear to be captured by the zig-zag Eulerian polynomials:
\vspace{0.5em}
\begin{itemize}
\setlength\itemsep{1em}

\item \textbf{Jacobi permutations}. Jacobi permutations, introduced by Viennot \cite{Viennot1980} in the context of Jacobi elliptic functions, are defined recursively. The empty word is a Jacobi permutation, as is any singleton permutation. For a set of integers $S$ with $m < \min S$, the Jacobi permutations of $S\cup \{m\}$, denoted $J_{S \cup \{m\}}$, are defined by $w = u(m)v$, where $u \in J_I$ for some even-sized subset $I$ and $v \in J_{S-I}$.\footnote{Our definition here actually gives a slight variant on Viennot's Jacobi permutations, as the permutations in Viennot's $J_n$ have the recursive form $u1v$ where $v$ is of even length.} See Table \ref{t-jacobi} for the Jacobi permutations in $J_n = J_{[n]}$ for $n\leq 5$, grouped according to the number of returns. Observe that all of the returns in Table \ref{t-jacobi} are big returns; it is straightforward to show that all returns among Jacobi permutations in $J_n$ are big.
\begin{table}
\[
 \begin{array}{c | ccccc}
 n & \ret(w)=0 & \ret(w)=1 & \ret(w)=2 & \ret(w) = 3 \\
 \hline
 \hline
 1 & 1 \\
 \hline
 2 & 12 \\
 \hline
 3 & 123 & 23\mathbf{1} \\
 \hline
 4 & 1234 & \begin{array}{c} 134\mathbf{2} \\ 23\mathbf{1}4 \\ 341\mathbf{2} \end{array} & 24\mathbf{1}\mathbf{3} \\
 \hline
 5 & 12345 & \begin{array}{c} 1245\mathbf{3} \\ 134\mathbf{2}5 \\ 1452\mathbf{3} \\ 23\mathbf{1}45 \\ 341\mathbf{2}5 \\ 4512\mathbf{3} \\ 2345\mathbf{1} \end{array} & \begin{array}{c} 135\mathbf{2}\mathbf{4} \\ 24\mathbf{1}\mathbf{3}5 \\ 25\mathbf{1}3\mathbf{4} \\ 351\mathbf{2}\mathbf{4} \\ 245\mathbf{3}\mathbf{1} \\ 34\mathbf{2}5\mathbf{1} \\ 452\mathbf{3}\mathbf{1} \end{array} & 35\mathbf{2}\mathbf{4}\mathbf{1}
 \end{array}
\]
\caption{All Jacobi permutations in $J_n$ up to $n=5$, grouped by the number of returns, which are highlighted in bold.} \label{t-jacobi}
\end{table}

\item \textbf{Decreasing binary trees}. A planar binary tree on $[n]$ is \emph{decreasing} if every child has a smaller label than its parent. When $n$ is odd, a \emph{complete} binary tree is one for which every non-leaf vertex has two children, and when $n$ is even, a binary tree is \emph{left complete} if every non-leaf vertex has two children except for the rightmost branch, which has a left leaf but no right leaf.

There is a simple bijection from alternating permutations $u \in U_n$ to such trees, given as follows. Observe that if $w \in U_n$, then $w = u(n)v$ where $u$ is an up-down permutation of odd size and $v$ an up-down permutation. We can therefore define the bijection recursively with $n$ at the root of the tree $T_w$, and $T_u$ as the left branch and $T_v$ the right branch. See Figure \ref{f-trees} for an example. Note that $i$ is a big return of $w$ precisely when, in $T_w$, $i$ is to the right of $i+1$ and is not a leaf adjacent to $i+1$.

\begin{figure}
\[
w=79451\mathbf{3}2\mathbf{8}\mathbf{6} \quad \longleftrightarrow \quad T_w=
  \begin{tikzpicture}[scale=.75,baseline=-1cm]
   \draw (0,0) node[fill=white,inner sep=2] (9) {$9$};
   \draw (6,-1) node[fill=white,inner sep=2,circle,draw=black] (8) {$8$};
   \draw (-1,-1) node[fill=white,inner sep=2] (7) {$7$}; 
   \draw (7,-2) node[fill=white,inner sep=2,circle,draw=black] (6)  {$6$};
   \draw (2,-2) node[fill=white,inner sep=2] (5)  {$5$};
   \draw (1,-3) node[fill=white,inner sep=2] (4)  {$4$};
   \draw (4,-3) node[fill=white,inner sep=2,circle,draw=black] (3)  {$3$};
   \draw (5,-4) node[fill=white,inner sep=2] (2)  {$2$};
   \draw (3,-4) node[fill=white,inner sep=2] (1)  {$1$};
   \draw (7)--(9)--(8)--(5)--(3)--(2);
   \draw (4)--(5);
   \draw (1)--(3);
   \draw (6)--(8);
  \end{tikzpicture}
\]
\caption{Example illustrating the bijection between alternating permutations and (left) complete binary trees. The big returns of $w$ are highlighted in bold; these correspond to the circled vertices in $T_w$.
} \label{f-trees}
\end{figure}
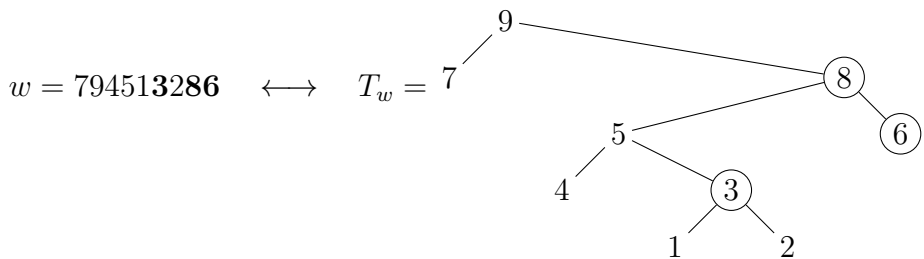

\end{itemize}

\vspace{0.5em}
It is possible that these interpretations or others might yield more insight into the polynomials $Z_n(t)$.

\subsection{More on relaxed $P$-partitions}\label{sec:morerelaxed}

As discussed in Section \ref{s-relaxed}, it is natural to define---for each $r\geq 0$---the \emph{$r$-relaxed $P$-partitions} as follows.

\begin{defn}[$r$-relaxed $P$-partition]
An \emph{$r$-relaxed $P$-partition} is an order-preserving map $g: P \to \NN$ such that for $i <_P j$:
\begin{itemize}
\item $g(i) \leq g(j)$ if $i \leq j+r$,
\item $g(i) < g(j)$ if $i > j+r$.
\end{itemize}
\end{defn}

We let $\A^r(P)$ denote the set of all $r$-relaxed $P$-partitions. When $r=0$, these are simply $P$-partitions. The $r=1$ case corresponds to the relaxed $P$-partitions used throughout this paper. Notice that we have
\[
 \A(P)=\A^0(P) \subseteq \A^1(P)\subseteq \A^2(P) \subseteq \A^3(P) \subseteq \cdots \subseteq \A^{\infty}(P),
\]
where $\A^{\infty}(P)$ denotes the set of all order-preserving maps from $P$ to $\NN$. In fact, $\A^r(P) = \A^{\infty}(P)$ for any $r \geq n-1$.

Much of the theory developed in Section \ref{s-relaxed} carries over readily for general $r\geq 0$. For example, the notion of a successive poset becomes that of an \emph{$r$-successive poset}, meaning that each element $i$ is comparable to each of $i+1, i+2, \ldots, i+r$. We have the following analogue of the fundamental lemma of $P$-partitions, generalizing Lemma~\ref{thm:ftpp} and Lemma~\ref{thm:ftpp2}.

\begin{lemma}[Fundamental lemma of $r$-relaxed $P$-partitions]\label{thm:rftpp2}
Suppose $P$ is $r$-successive. Then the set of $r$-relaxed $P$-partitions is the disjoint union of the $\pi$-partitions of its linear extensions $\pi$:
\[
 \A^r(P) = \bigcup_{\pi \in \LL(P)} \A^r(\pi).
\]
\end{lemma}

This statement implies the corresponding order polynomials $\op^r_P(m)$ can be expressed as a sum over linear extensions and we get the following generalization of Proposition \ref{prp:decomp}. 

\begin{prop}\label{prp:rdecomp}
For any $r$-successive poset $P$ of $[n]$, we have
\[
 \sum_{m\geq 1} \op^r_P(m) t^m = \frac{\sum_{\pi \in \LL(P)} t^{1+\des_r(\pi)}}{(1-t)^{n+1}},
\]
where we recall from Section \ref{sec:alt} that $\des_r(w)=|\{\, i : w(i)>w(i+1)+r \,\}|$.
\end{prop}

The technology of $r$-relaxed $P$-partitions can be used to deduce many known results of Foata and Sch\"utzenberger \cite{FS} for the distribution of $\des_r$ over the full symmetric group.

The rank $r$ generalization of the zig-zag poset is what we call the \emph{$r$-chainlink poset}, denoted $Z_n^r$. This is the poset on $[n]$ defined in terms of division by $r+1$ as follows. Let $i = a(r+1)+ b$ and $j = c(r+1)+d$, where $a,c\geq 0$ and $0\leq b,d \leq r$. Then $i <_{Z^r_n} j$ if and only if $a \geq c$ and $b \leq d$. This is illustrated in Figure \ref{f-chainlink} for $r=3$ and $n=22$. Note that $Z^r_n$ is a rank $r$ lattice with ranks corresponding to congruence classes modulo $r+1$.

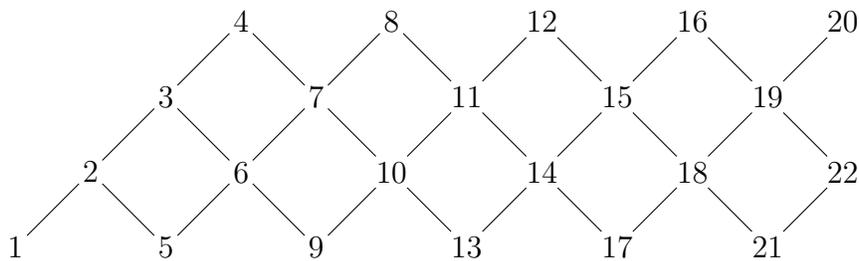
\begin{figure}
\[
\begin{tikzpicture}
\draw (0,0)--(3,3);
\draw (2,0)--(5,3);
\draw (4,0)--(7,3);
\draw (6,0)--(9,3);
\draw (8,0)--(11,3);
\draw (10,0)--(11,1);
\draw (1,1)--(2,0);
\draw (2,2)--(4,0);
\draw (3,3)--(6,0);
\draw (5,3)--(8,0);
\draw (7,3)--(10,0);
\draw (9,3)--(11,1);
\draw (0,0) node[fill=white, circle, inner sep=1] {$1$};
\draw (1,1) node[fill=white, circle, inner sep=1] {$2$};
\draw (2,2) node[fill=white, circle, inner sep=1] {$3$};
\draw (3,3) node[fill=white, circle, inner sep=1] {$4$};
\draw (2,0) node[fill=white, circle, inner sep=1] {$5$};
\draw (3,1) node[fill=white, circle, inner sep=1] {$6$};
\draw (4,2) node[fill=white, circle, inner sep=1] {$7$};
\draw (5,3) node[fill=white, circle, inner sep=1] {$8$};
\draw (4,0) node[fill=white, circle, inner sep=1] {$9$};
\draw (5,1) node[fill=white, circle, inner sep=1] {$10$};
\draw (6,2) node[fill=white, circle, inner sep=1] {$11$};
\draw (7,3) node[fill=white, circle, inner sep=1] {$12$};
\draw (6,0) node[fill=white, circle, inner sep=1] {$13$};
\draw (7,1) node[fill=white, circle, inner sep=1] {$14$};
\draw (8,2) node[fill=white, circle, inner sep=1] {$15$};
\draw (9,3) node[fill=white, circle, inner sep=1] {$16$};
\draw (8,0) node[fill=white, circle, inner sep=1] {$17$};
\draw (9,1) node[fill=white, circle, inner sep=1] {$18$};
\draw (10,2) node[fill=white, circle, inner sep=1] {$19$};
\draw (11,3) node[fill=white, circle, inner sep=1] {$20$};
\draw (10,0) node[fill=white, circle, inner sep=1] {$21$};
\draw (11,1) node[fill=white, circle, inner sep=1] {$22$};
\end{tikzpicture}
\]
\caption{The $r$-chainlink poset $Z^r_n$ for $r=3$ and $n=22$.} \label{f-chainlink}
\end{figure}

We define 
\[
Z^r_n(t) = \sum_{\pi \in \LL(Z_n^r)} t^{1+\des_r(\pi)} = \sum_{k=0}^{n-r-1}z^r(n,k)t^{k+1}.
\]
Notice that $Z^0_n$ is an antichain and $Z^1_n$ is the zig-zag poset $Z_n$, and so $Z^0_n(t) = A_n(t)$ is the classical Eulerian polynomial while $Z^1_n(t) = Z_n(t)$ is the zig-zag Eulerian polynomial studied in this paper. See Table \ref{tab:des2} for the coefficients of $Z_n^2(t)$ and $Z_n^3(t)$.

Using techniques analogous to those in Section \ref{sec:proof1}, we can show $Z^r_n(t)$ equals the $P$-Eulerian polynomial for a natural labeling of $Z_n^r$, and hence we can apply Br\"and\'en's result again to achieve the following.

\begin{thm} \label{t-gamma2}
Fix $r\geq 0$. For each $n \geq 1$, there exist nonnegative integers $\gamma_{n,j}$ such that
\[
 Z_n^r(t) = t\cdot\sum_{0\leq 2j\leq n-r-1} \gamma_{n,j} t^j(1+t)^{n-r-1-2j}.
\]
\end{thm}

\begin{table}
\[
\begin{array}{cc}
\begin{array}{|c|cccccccc|}
\hline
n\backslash k & 0 & 1 & 2 & 3 & 4 & 5 & 6 & 7\\
\hline
1 & 1 &&&&&&&\\
2 & 1 &&&&&&&\\
3 & 1 &&&&&&&\\
4 & 1 & 1 &&&&&&\\
5 & 1 & 3 & 1 &&&&&\\
6 & 1 & 6 & 6 & 1 &&&&\\
7 & 1 & 10 & 20 & 10 & 1 &&& \\
8 & 1 & 18 & 65 & 65 & 18 & 1 && \\
9 & 1 & 30 & 181 & 320 & 181 & 30 & 1 & \\
10 & 1 & 43 & 389 & 1097 & 1097 & 389 & 43 & 1 \\
\hline
 \end{array}
 &
 \begin{array}{|c|ccccccc|}
\hline
n\backslash k & 0 & 1 & 2 & 3 & 4 & 5 & 6 \\
\hline
1 & 1 &&&&&&\\
2 & 1 &&&&&&\\
3 & 1 &&&&&&\\
4 & 1 &  &&&&&\\
5 & 1 & 1 & &&&& \\
6 & 1 & 3 & 1 &  &&& \\
7 & 1 & 6 & 6 & 1 & &&\\
8 & 1 & 10 & 20 & 10 & 1 & & \\
9 & 1 & 14 & 46 & 46 & 14 & 1 &  \\
10 & 1 & 22 & 113 & 190 & 113 & 22 & 1  \\
\hline
 \end{array}
\\
\rule{0pt}{3ex}
r=2 & r=3
\end{array}
\]
\caption{Table of values of $z^r(n,k)$ for $r\in \{2,3\}$, $1\leq n \leq 10$, and \\ $0\leq k \leq n-r-1$.}\label{tab:des2}
\end{table}

Recall that the \emph{Narayana numbers} count $231$-avoiding permutations according to descent number. See, e.g., Chapter 2 of \cite{Petersen2015}. Let
\[
 C_n(t) = \sum_{w\in S_n(231)} t^{\des(w)},
\]
where $S_n(231)$ is the subset of $231$-avoiding permutations in $S_n$, be the $n$th \emph{Narayana polynomial}. A quick calculation shows $C_3(t)=1+3t+t^2$, and we can see above that $tC_3(t)= Z_4^1(t) = Z_5^2(t) = Z_6^3(t)$. Through a bijection between $231$-avoiding permutations and linear extensions of the poset $[2]\times [n]$ that preserves the descent number statistic, we can make the following observation.

\begin{obs}
 For all $1\leq j \leq r+1$, we have $Z_j^r(t) = 1$, and for all $1\leq j \leq r+2$ we have
 \[
  Z_{r+j}^r(t) = tC_j(t).
 \]
\end{obs}

This is an interesting sort of stability, and we can see that while the classical Eulerian polynomials (arising from the study of the permutahedron) come from $Z_n^0(t)$, we also have the Narayana polynomials (arising from the study of the associahedron) coming from $Z_{r+j}^r(t)$. That both families occur within the larger family $Z_n^r(t)$ is intriguing.

Finally, given the real-rootedness of the classical Eulerian polynomials and our conjectured real-rootedness of the zig-zag Eulerian polynomials $Z_n(t)$, we pose the following question.

\begin{question}
Are the polynomials $Z_n^r(t)$ real-rooted for all $r\geq 0$ and $n\geq 1$?
\end{question} 

\section*{Acknowledgements}
We give thanks to the On-Line Encyclopedia of Integer Sequences (OEIS) for giving us the breadcrumbs to follow when researching this problem. Thanks also to Rafael Gonz\'alez D'Le\'on and Persi Diaconis for helpful conversation and references. Finally, we thank two anonymous referees for corrections and suggestions which improved the presentation of this paper. YZ was partially supported by NSF grant DMS-2316181.

\bibliographystyle{plain}
\bibliography{bibliography}

\end{document}